\documentclass[11pt]{amsart}

\usepackage[utf8]{inputenc}
\usepackage[T1]{fontenc}
\usepackage[french, english]{babel}

\usepackage{amsaddr}

\usepackage{amsthm}
\usepackage{amsmath}
\usepackage{amssymb}
\usepackage{mathrsfs}
\usepackage{dsfont}
\usepackage{graphicx}
\usepackage{listings}
\usepackage{mathrsfs}
\usepackage{enumerate}
\usepackage{pict2e}
\usepackage{stmaryrd}
\usepackage{mathtools}
\usepackage{subcaption}

\usepackage[all,cmtip]{xy}

\usepackage{multirow}

\usepackage{color}
\definecolor{marin}{rgb}   {0.,   0.1,   0.5} 
\definecolor{rouge}{rgb}   {0.8,   0.,   0.} 
\definecolor{sepia}{rgb}   {0.4,   0.25,   0.} 
\definecolor{mag}{rgb}   {0.3,   0,   0.3} 

\usepackage[colorlinks,citecolor=sepia,linkcolor=marin,urlcolor=mag ,bookmarksopen,bookmarksnumbered]{hyperref}

\newcommand{\R}{\ensuremath{\mathbb{R}}}

\newcommand{\abs}[1]{\ensuremath{\left|#1\right|}}

\newtheorem{theorem}{Theorem}[section]
\newtheorem{corollary}[theorem]{Corollary}
\newtheorem{lemma}[theorem]{Lemma}
\newtheorem{proposition}[theorem]{Proposition}
\newtheorem{definition}[theorem]{Definition}
\newtheorem{remark}[theorem]{Remark}

\begin{document}

\title[Almost conservation of the harmonic actions at low regularity
]{Almost conservation of the harmonic actions for fully discretized nonlinear Klein--Gordon equations at low regularity}

\author{Charbella Abou Khalil and Joackim Bernier}

\email{Charbella.AbouKhalil@univ-nantes.fr}

\address{\small{Nantes Universit\'e, CNRS, Laboratoire de Math\'ematiques Jean Leray, LMJL,
F-44000 Nantes, France
}}

\email{joackim.bernier@univ-nantes.fr}

\keywords{Splitting methods, backward error analysis, Birkhoff normal form, low regularity}

\subjclass[2010]{65P10, 65P40, 37K55, 35B34}

\begin{abstract}
Close to the origin, the nonlinear Klein--Gordon equations on the circle are nearly integrable Hamiltonian systems which have infinitely many almost conserved quantities called harmonic actions or super-actions. We prove that, at low regularity and with a CFL number of size $1$, this property is preserved if we discretize the nonlinear Klein--Gordon equations
with the symplectic mollified impulse methods. This extends previous results of D. Cohen, E. Hairer and C. Lubich to non-smooth solutions.
\end{abstract}
\maketitle

\setcounter{tocdepth}{1} 
\tableofcontents

\section{Introduction}

\subsection{Context and motivation} We consider the \emph{nonlinear Klein--Gordon} equations on the torus $\mathbb{T}:=\mathbb{R}\slash 2\pi \mathbb{Z}$ (i.e. on the segment $[0,2\pi]$ with periodic boundary conditions)
\begin{align}\label{N1}\tag{KG}
  \partial_{t}^2q -\partial_{x}^2q+\rho q+g(q)=0
    \end{align}
where $q=q(x,t)\in \R$ with $(x,t)\in \mathbb{T}\times \R$, $\rho >0$ is a parameter called \emph{mass} and $g\in C^{\infty}(\mathbb{R};\mathbb{R})$ is a smooth real function having a zero of order $2$ at $0,$ i.e. satisfying $g(0)=g'(0)=0$. 

\medskip

They are \emph{Hamiltonian partial differential equations}. Indeed, setting, as usual, 
$$
p := \partial_t q,
$$
the nonlinear Klein--Gordon equation \eqref{N1} rewrites
\[\partial_t 
\begin{pmatrix}
q\\
p
\end{pmatrix} = \begin{pmatrix}
0 & 1 \\
\partial_x^2 - \rho & 0
\end{pmatrix}\begin{pmatrix}
q\\
p
\end{pmatrix} - \begin{pmatrix}
0\\
g(q)
\end{pmatrix}=\begin{pmatrix}
0 & 1 \\
-1 & 0
\end{pmatrix} \nabla H(q,p) \] with 
$$
H(q,p):= \int_{\mathbb{T}}\Big(\frac{1}{2} (p^2+(\partial_x q)^2 + \rho q^2 ) + G(q)\Big) \,\frac{\mathrm{d}x}{2\pi} 
$$ 
where $G \in C^\infty(\mathbb{R};\mathbb{R})$ is the primitive of $g$ vanishing at the origin and $\nabla=(\nabla_q,\nabla_p)$ denotes the standard gradient for the  $L^2$ scalar product. Note that the Hamiltonian $H$ is a constant of motion of \eqref{N1}.

\medskip

We are interested in the dynamics of the \emph{small solutions} of \eqref{N1}. In this regime, \eqref{N1} is a nearly integrable Hamiltonian system. Indeed, since $g(0) = g'(0)$, in the Fourier variables 
$$
q_k := \int_{\mathbb{T}} q(x) e^{-ikx} \frac{\mathrm{d}x}{2\pi}, \quad p_k := \int_{\mathbb{T}} p(x) e^{-ikx} \frac{\mathrm{d}x}{2\pi}, \quad \mathrm{for} \ k\in \mathbb{Z,}
$$
for small solutions, \eqref{N1} is a perturbation of the linear integrable system
\begin{equation}
\label{eq:kglin}
\partial_t 
\begin{pmatrix}
q_k\\
p_k
\end{pmatrix} = \begin{pmatrix}
0 & 1 \\
-\omega_k^2 & 0
\end{pmatrix}\begin{pmatrix}
q_k\\
p_k
\end{pmatrix}, \quad k\in \mathbb{Z}
\end{equation}
of frequencies
$$
\omega_k := \sqrt{k^2 +\rho}.
$$
This linear equation \eqref{eq:kglin} has infinitely many constants of motion. Among all of them, we expect, from perturbation theory in the finite dimensional setting (and in particular Birkhoff normal forms; see \cite{MR2840298}), that, for typical values of the mass $\rho$, the \emph{harmonic actions} (also called \emph{super-actions})
$$
J_k(p,q) :=  \omega_k^{-1} |p_k|^2 + \omega_k |q_k|^2 
$$
are almost preserved, for very long times, by the flow of \eqref{N1}.

\medskip

 More precisely, we consider solutions of \eqref{N1} generated by initial data
 $$
 (q(0),p(0)) = (q^{(0)},p^{(0)})
 $$ 
 of size $\varepsilon \ll 1$ in a Sobolev space $H^{s+1/2}(\mathbb{T};\mathbb{R})\times H^{s-1/2}(\mathbb{T};\mathbb{R})$ for some $s>0$, i.e. 
$$
\varepsilon = \Vert (q^{(0)},p^{(0)}) \Vert_{H^{s+1/2} \times H^{s-1/2}} \ll 1.
$$
The Sobolev spaces $H^s$ are defined by
$$
H^s(\mathbb{T};\mathbb{K}) :=\{  u\in L^2(\mathbb{T};\mathbb{K}) \ | \ \| u\|_{H^s}^2 = \sum_{k\in \mathbb{Z}} \langle k \rangle^{2s}|u_k|^2<\infty \}, \quad \mathbb{K}\in \{ \mathbb{R},\mathbb{C}\}
$$
where the numbers $u_k := (2\pi)^{-1}\int_{\mathbb{T}} u(x) e^{-ikx} \mathrm{d}x$ denote the Fourier coefficients of $u$ and $\langle \cdot \rangle := \sqrt{1+|\cdot|^2}$ denotes the Japanese bracket.

\medskip

In all the paper, given two real numbers $a,b\in \mathbb{R}$, the notation $a \lesssim_p b$ means that there exists a constant $C>0$ depending only on $p$ (and possibly other constants clearly considered as fixed) such that $a \leq C b$.  We write $a\sim_p b$ if $a\lesssim_p b$ and $a \lesssim_p b$.

\medskip

The nonlinear Klein--Gordon equation \eqref{N1} is locally well-posed in $H^{s+1/2}(\mathbb{T};\mathbb{R})\times H^{s-1/2}(\mathbb{T};\mathbb{R})$ for any $s>0$. As a consequence, by a simple homogeneity argument, it can be proven that before times of order $\varepsilon^{-1}$, the nonlinear effects are negligible with respect to the linear ones\footnote{i.e. the solution of the nonlinear equation remains close to the solution of the linear one}, and so that the harmonic actions are almost preserved, i.e.
$$
t \ll \varepsilon^{-1} \implies \sum_{k\in \mathbb{Z}} \langle k\rangle^{2s+1} \big| J_k(p(t),q(t)) -J_k(p(0),q(0)) \big| \ll \varepsilon^2 \sim \sum_{k\in \mathbb{Z}} \langle k\rangle^{2s} J_k(p(0),q(0)).
$$ 
In this setting, the problem is then to prove the almost preservation of the harmonic actions for non trivial times, i.e. for times much larger than $\varepsilon^{-1}$. In high regularity, this problem has been solved by Bambusi \cite{Bam03}, Bambusi--Gr\'ebert \cite{BG06} and Cohen--Hairer--Lubich \cite{MR2366141}. They proved, for typical values of the mass $\rho$, the almost preservation of the harmonic actions for times of order $\varepsilon^{-r}$, for any $r$ arbitrarily large,  provided that the initial datum is of size $\varepsilon \ll 1$ in $H^{s+1/2}\times H^{s-1/2}$ with $s$ large enough with respect to $r$. More precisely, they proved that for all $r\geq 1$, there exists $s_0(r) \geq 1$ such that for almost all $\rho>0$, provided that $s\geq s_0(r)$ and $\varepsilon = \Vert (q^{(0)},p^{(0)}) \Vert_{H^{s+1/2} \times H^{s-1/2}} \lesssim_{r,s,\rho} 1 $, the associated solution to \eqref{N1} satisfies
\begin{equation}
\label{eq:classicul}
t \leq \varepsilon^{-r} \implies \sum_{k\in \mathbb{Z}} \langle k\rangle^{2s+1} \big| J_k(p(t),q(t)) -J_k(p(0),q(0)) \big| \lesssim_{r,s,\rho} \varepsilon^3.
\end{equation}
As discussed in \cite{MR2366141}, in these proofs, $s_0(r)$ is very large even for quite small values of $r$. Also in these proofs, $s_0(r)$ goes to $+\infty$ as $r$ goes to $+\infty$ (at least like $r^2$). The smoothness assumption $s\geq s_0(r)$, where $s_0(r)$ goes to $+\infty$ as $r$ goes to $+\infty$, is crucial in the proofs of this result in order to deal with the singularities generated by the small divisors. 

\medskip

In a series of papers \cite{MR2366141,MR2413147,zbMATH05323311}, D. Cohen, E. Hairer and C. Lubich proved that this almost conservation property survives to some full discretizations given by some standard symplectic integrators with CFL number of order $1$ (described just below). Moreover, they included the results of some numerical simulations suggesting that, surprisingly, even in low regularity, the harmonic actions are almost preserved (or, in other words, that $s_0(r)$ should be small and should not depend on $r$). Motivated by these numerical observations, the second author and B. Gr\'ebert, introduced, in \cite{bernier2021birkhoff}, new small divisors estimates (recalled in Section \ref{z1}) and a new partial Birkhoff normal form to prove the almost preservation of the \emph{low} harmonic actions at low regularity (in the energy space). More precisely, they proved that for all $r\geq 1$, there exists $\beta_r \geq 1$ such that for almost all $\rho>0$, provided that $s=1/2$ and $\varepsilon = \Vert (q^{(0)},p^{(0)}) \Vert_{H^{1} \times L^2} \lesssim_{r,\rho} 1 $, the associated solution to \eqref{N1} satisfies
\begin{equation}
\label{eq:low_reg_poly}
t \leq \varepsilon^{-r} \implies \forall k\in \mathbb{Z}, \ \big| J_k(p(t),q(t)) -J_k(p(0),q(0)) \big| \lesssim_{r,\rho} \langle k \rangle^{\beta_r} \varepsilon^3.
\end{equation}
The smoothness assumption $s\geq s_0(r)$ is removed but, due to the growing factor $ \langle k \rangle^{\beta_r}$, we only control the variation of the low harmonic actions.

\medskip

The main result of this paper consists in proving that this almost preservation property at low regularity survives to the full discretizations considered by Cohen, Hairer and Lubich in \cite{zbMATH05323311}. However, contrary to \cite{zbMATH05323311}, our proof does not rely on modulated Fourier expansions but on backward error analysis and Birkhoff normal form (we refer to subsection \ref{sub:proof} for discussions about the proof). 

\subsection{Discretization} \label{sub:disc} We consider the same discretization as D. Cohen, E. Hairer and C. Lubich in \cite{zbMATH05323311}. However, to ensure that this paper is self-contained, we recall it.

\medskip

\noindent \emph{\underline{Semi-discretization.}} With respect to the space variable, we consider the standard pseudo-spectral discretization (with aliasing). The unknowns $q(t),p(t) : \mathbb{T}_K \to \mathbb{R}$ are real-valued functions on the discretized torus, $\mathbb{T}_K$, with $K$ equidistant points
$$
\mathbb{T}_K := \frac{2\pi}{K} (\mathbb{Z}/K\mathbb{Z}).
$$
We define the discrete Fourier coefficients of any complex-valued function $u : \mathbb{T}_K \to \mathbb{C}$ on the discrete torus by
$$
u_k = \frac1{K} \sum_{x\in \mathbb{T}_K} u(x) e^{-i k x}, \quad k\in \mathcal{N}_K :=  [-K/2,K/2)\cap \mathbb{Z}.
$$
Note that the inverse Fourier transform formula ensures that for all $x\in \mathbb{T}_K$
\begin{equation}
\label{eq:fourier_inverse}
u(x) = \sum_{k\in \mathcal{N}_K} u_k e^{i kx}.
\end{equation}
As a consequence, we identify $\mathbb{C}^{\mathcal{N}_K}$ with the space of the complex-valued functions on $\mathbb{T}_K$. Moreover, the formula \eqref{eq:fourier_inverse} also allows to identify the functions on the discrete torus $\mathbb{T}_K$ with trigonometric polynomials on $\mathbb{T}$ whose Fourier coefficients are in $\mathcal{N}_K$. This identification provides a definition of $\partial_x$ as an operator acting on $\mathbb{C}^{\mathcal{N}_K}$ by the formula
$$
\forall u \in\mathbb{C}^{\mathcal{N}_K},\forall k\in \mathcal{N}_K, \ (\partial_x u)_k := ik u_k. 
$$
We recall that the discrete Fourier transform and the pseudo spectral interpolation are isometries in the sense that\footnote{these formula are also true for the Hermitian scalar product (i.e. without the real part), but only the real scalar product is useful in this paper.} for all complex-valued functions  $u,v \in \mathbb{C}^{\mathcal{N}_K}$ on the discrete torus $\mathbb{T}_K$,
\begin{equation}
\label{eq:isom}
(u,v)_{L^2} := \Re \sum_{k\in \mathcal{N}_K} u_k \overline{v_k} =\frac1K \Re \sum_{x\in \mathbb{T}_K} u(x) \overline{v(x)} = \Re \int_{\mathbb{T}} u(x) \overline{v(x)} \frac{\mathrm{d}x}{2\pi} . 
\end{equation}
As a consequence, the semi-discretized nonlinear Klein--Gordon equation can be written in similar form as the continuous one, i.e.
\begin{equation}
\label{eq:KG_sd}
\tag{$\mathrm{KG}_{\mathrm{sd}}$}
\partial_t 
\begin{pmatrix}
q\\
p
\end{pmatrix} = \begin{pmatrix}
0 & 1 \\
\partial_x^2 - \rho & 0
\end{pmatrix}\begin{pmatrix}
q\\
p
\end{pmatrix} - \begin{pmatrix}
0\\
g(q)
\end{pmatrix}.
\end{equation}
The only difference is that now $p,q : \mathbb{T}_K \to \mathbb{R}$ are real-valued functions on the discrete torus $\mathbb{T}_K$. Moreover, the semi-discretized Klein--Gordon equation \eqref{eq:KG_sd} is also a Hamiltonian system. Indeed, it rewrites
$$
\partial_t 
\begin{pmatrix}
q\\
p
\end{pmatrix} = \begin{pmatrix}
0 & 1 \\
-1 & 0
\end{pmatrix} \nabla H^{K}(q,p)
$$
 where $\nabla = (\nabla_q,\nabla_p)$, $\nabla_q,\nabla_p$ are the partial gradients for the $L^2$ scalar product and
$$
H^K(q,p):= T^K(q,p) + W^K(q)
$$ 
with 
\begin{equation}
\label{eq:def_TK} 
 T^K(q,p)  :=  \int_{\mathbb{T}}\frac{1}{2}\left(p^2+(\partial_x q)^2 + \rho q^2 \right) \frac{\mathrm{d}x}{2\pi}   =\frac12 \sum_{k\in \mathcal{N}_K} \omega_k J_k(q,p)  .
\end{equation}
and
\begin{equation}
\label{eq:def_WK}
W^K(q):= \frac1K \sum_{x\in \mathbb{T}_K} G(q(x)).
\end{equation}
 For readability of notations, from now, we drop the discretization exponent $K$ throughout the article. Nevertheless, we will always pay attention to establish uniform estimates with respect to $K$ (i.e. the implicit constants in the notations $\lesssim$ do not depend on $K$). 
 
\medskip 
 
\noindent  \emph{\underline{Full discretization.}} For the time discretization, as in  \cite{zbMATH05323311}, we consider a \emph{symplectic mollified impulse method}. These methods have been introduced in \cite{GSS98}. They are Strang splitting methods with mollifiers. To present the method, we have to introduce the Fourier multiplier $\Lambda =\sqrt{-\partial_x^2+\rho}$, naturally defined by
$$
\forall u \in \mathbb{C}^{\mathcal{N}_K},\forall k\in \mathcal{N}_K,  \ (\Lambda u)_k := \omega_k \, u_k.
$$
In this paper, we consider the following full discretization of the Klein--Gordon equation
\begin{equation}
\label{eq:KG_fd}
\tag{$\mathrm{KG}_{\mathrm{fd}}$}
\left\{ \begin{array}{lll}
 q^{n+1} = \cos(h\Lambda)q^{n} + h\, \text{sinc}(h\Lambda)p^{n} -\frac{h^2}{2}\text{sinc}(h\Lambda)\phi(h\Lambda)g(\phi(h\Lambda)q^{n} )  \vspace{0.3cm} \\
 p^{n+1} = -\Lambda\sin(h\Lambda)q^{n} + \cos(h\Lambda)p^{n}  -\frac{h}{2}  \phi(h\Lambda)\left( \cos(h\Lambda) g(\phi(h\Lambda)q^{n} )+ g(\phi(h\Lambda)q^{n+1} ) \right)
\end{array}\right.
\end{equation}
where $h>0$ is the \emph{time step}, $q^n,p^n : \mathbb{T}_K \to \mathbb{R}$ are real-valued functions on the discrete torus $\mathbb{T}_K$, $n\in \mathbb{N}$ is an integer such that $(q^n,p^n)$ is designed to be an approximation of $(q(hn),p(hn))$, $\text{sinc}(X) =\sin(X)/X$ is the usual cardinal sine function, and $\phi \in C^\infty(\mathbb{R}_+;\mathbb{R})$ is a smooth bounded  real-valued function satisfying $\phi(0) = 1$.  We consider the function $\phi$ as given once and for all, we do not try to track the dependencies of the parameters with respect to $\phi$.

\medskip

This method also rewrites as the following $2$-steps method
$$
q^{n+1}-2 \cos(h \Lambda) q^n + q^{n-1} =- h^2 \phi(h\Lambda) \, \mathrm{sinc}( h\Lambda) \, g( \phi(h\Lambda)q^{n} ).
$$
It is also very useful, both for the understanding of the method and its analysis, to notice that it is a Strang splitting with a mollifier. Indeed, 
tedious but direct calculations show that \eqref{eq:KG_fd} rewrites
$$
(q^{n+1},p^{n+1}) = \Phi_{\mathrm{num}}^h(q^n,p^n)
$$
with the numerical flow $\Phi_{\mathrm{num}}^h$ given by the Strang splitting
$$
\Phi_{\mathrm{num}}^h := \Phi^{h/2}_{V}\circ \Phi^{h}_{T} \circ \Phi^{h/2}_{V},
$$
where $\Phi_{V}$,$\Phi_{T}$ denote respectively the flows of the Hamiltonian equations
\begin{equation}
\label{eq:flowVqp}
\partial_t \begin{pmatrix} q\\p \end{pmatrix} = \begin{pmatrix}
0 & 1 \\
-1 & 0
\end{pmatrix} \nabla T(q,p) \quad \mathrm{and} \quad \partial_t \begin{pmatrix} q\\p \end{pmatrix} = \begin{pmatrix}
0 & 1 \\
-1 & 0
\end{pmatrix} \nabla V(q,p) 
\end{equation}
and $V$ is a mollified version of $W$ defined by
\begin{equation}\label{mollified_perturbation}
V(q,p)  = W( \phi(h\Lambda) q).
\end{equation}
Note that thanks to this formulation as splitting method, it is clear that the numerical flow of \eqref{eq:KG_fd} is symplectic.

\subsection{Main result}

Now, we are in position to state the main result of this paper.
\begin{theorem} For \label{thm:main} all $r\geq 1$, there exists $\beta_r \geq 1$, such that for almost all $\rho >0$, all $\delta>0$, all $K\geq 1$ and all $h>0$ satisfying 
\begin{equation}
\label{eq:CFL}
(r+2) \, h\, \omega_{K/2} \leq 2\pi -\delta,
\end{equation}
 there exists $\varepsilon_0(r,\rho)>0$ depending only on $r$ and $\rho$ such that if $(q^n,p^n)_{n\geq 0}$ is a sequence of real-valued functions on $\mathbb{T}_K$ solution of the fully discretized nonlinear Klein--Gordon equation \eqref{eq:KG_fd} and whose initial datum is small enough in the energy space $H^1 \times L^2$, i.e.
 $$
 \varepsilon := \| q^0 \|_{H^1} + \| p^0 \|_{L^2} \leq \varepsilon_0(r,\rho)
 $$
 then the low harmonic actions are almost preserved for times of order $\varepsilon^{-r}$, i.e.
\begin{equation}
\label{eq:principal_truc}
nh \leq \varepsilon^{-r} \implies \forall k\in \mathcal{N}_K, \ |J_k(q^n,p^n)-J_k(q^0,p^0) | \lesssim_{\delta,r} \langle k \rangle^{\beta_r} \varepsilon^3.
\end{equation}
\end{theorem}
This result partially explains the observations of D. Cohen, E. Hairer and C. Lubich in \cite{MR2366141,MR2413147,zbMATH05323311} and the one presented in subsection \ref{sub:exp} below. Indeed, the smoothness assumption has been removed, the initial data only have to be small in $H^1 \times L^2$ which is a regularity lower than the one they consider in their simulations. Theorem \ref{thm:main} only ensures the almost global preservation of the low harmonic actions. Nevertheless, in their simulations, they observe that the high harmonic actions seem also almost preserved.  This last phenomena is not explained by Theorem \ref{thm:main}.

\medskip

\noindent \emph{\underline{Comments about the CFL condition \eqref{eq:CFL} .}} The CFL condition \eqref{eq:CFL}  imposes a constraint on the numerical parameters $h,K$ which is uniform with respect to the size $\varepsilon$ of the initial datum. This allows to approximate solutions of  \eqref{N1} with arbitrary precision, i.e. the condition \eqref{eq:CFL} allows $h$ to go to $0$ and $K$ to $+\infty$ (contrary, for example, to the conditions introduced in \cite{MR2570074}).Note that in particular, the estimate on the right hand side of \eqref{eq:principal_truc} is uniform with respect to the numerical parameters $K$ and $h$.  Recalling that $\omega_k \sim \langle k\rangle$, we note that this is a CFL of  transport type (i.e. linear in $K$ and $h$) and so that it is not too costly. It allows to deal with the resonances generated by the time discretization of the equation. More precisely, it provides a uniform upper bound on non trivial quantities of the form
\begin{equation}
\label{eq:relou}
\big(1-e^{ih (\pm \omega_{k_1} \pm \cdots \pm \omega_{k_m})} \big)^{-1}
\end{equation}
with $m\leq r+2$ and $k_1,\cdots,k_m \in \mathcal{N}_K$.  These quantities are related to the problem of the resonant time steps and seems unavoidable in the analysis of the dynamics of discretized equations (see e.g. \cite{MR2895408,MR2840298}).

\medskip

The CFL condition \eqref{eq:CFL}  is not related to the low regularity issues. It also appears in the almost conservation of the Hamiltonian or the almost conservations of some harmonic actions in high regularity (see \cite{MR2811583,MR2895408,MR3712186}). As in \cite{MR3712186}, we could generalize it a little by imposing directly an upper bound on \eqref{eq:relou} by a negative small power of $\varepsilon$. When the CFL condition is violated, up to a genericity assumption on $h$, this is almost equivalent to imposing an upper bound on $K$ by a small negative power of $\varepsilon$. As explained in \cite{MR3712186}, up to such a generalization, the CFL condition \eqref{eq:CFL} covers the same kind of regimes as the non resonances conditions of Cohen--Hairer--Lubich in  \cite{zbMATH05323311} (but is simpler).

\medskip

\noindent \emph{\underline{Comments about the choice of the numerical method.}} Since we were motivated by the numerical observations in \cite{MR2366141,MR2413147,zbMATH05323311}, we chose to focus on the same numerical methods. With a few more notations and technicalities, the results could be extended to general splitting methods. Actually, for simplicity, we prove the result for the numerical flow associated with the Lie splitting
$$
\Phi_{\mathrm{num}}^h = \Phi^{h}_{V}\circ \Phi^{h}_{T},
$$
and then we use that the Strang splitting is conjugated to the Lie one. It seems to us that only two properties of the method are really crucial in our proof. The first one is that the numerical flow is symplectic. The second one, more technical, is that the modified frequencies (given by the backward error analysis) satisfy the strong non-resonance conditions introduced in \cite{bernier2021birkhoff} (and recalled in Proposition \ref{prop:non_res}).
Here this second property is ensured by the pseudo-spectral semi-discretization (which ensure that semi-discretized frequencies are a subset of the original ones) and the fact that the linear part of the semi-discretized equation is solved exactly by the numerical flow, i.e.
$$
\mathrm{d}\Phi_{\mathrm{num}}^h(0) = \mathrm{d}\Phi_{H}^h(0) = \exp\left( h \begin{pmatrix}
0 & 1 \\
\partial_x^2 - \rho & 0
\end{pmatrix} \right).
$$ 
We refer to \cite{demande_reviewer} for results concerning conservation properties of more general multi-steps methods and not necessarily small initial data.

The mollified impulse method is probably not the most accurate method to approximate solutions to \eqref{N1} in the energy space $H^1 \times L^2$. Indeed, classical numerical methods tend to suffer order reduction when applied to nonlinear dispersive equations at low regularity. A lot of progress have been made in recent years in analyzing these phenomena and developing new methods to avoid them (see e.g. \cite{MR4567995,WZ22,BS22,RS21,Ala23,OS18,ORS23,MS23}).
In any case, Theorem \ref{thm:main} allows to consider initial data smoother than $H^1\times L^2$ (e.g. in $H^2\times H^1$) and so to avoid order reduction issues.

\medskip

\noindent \emph{\underline{Comments about the choice of the equation.}} We focused on \eqref{N1} for simplicity to be in the same setting as Cohen--Hairer--Lubich in \cite{MR2366141,MR2413147,zbMATH05323311}. Nevertheless, the result and the proof are robust, and we expect that they could be applied to other models. At the continuous level (i.e. without any discretization), results similar to Theorem \ref{thm:main}, at low regularity, have also been proven for nonlinear Schr\"odinger equations in dimension $d\leq 2$ \cite{bernier2021birkhoff,bernier:hal-04090717}, for the nonlinear Klein--Gordon equation on the sphere in \cite{BGR21} and for the Gross-Pitaevskii equation on $\mathbb{R}$ in \cite{Charb}. The main limitation is that, for the moment, to prove that the non-resonance conditions are satisfied, we have to consider models for which the frequencies converge to the integer, i.e.
\begin{equation}
\label{eq:main_lim}
 \inf_{n\in \mathbb{Z}} |n-\omega_k| \mathop{\longrightarrow}_{k\to +\infty} 0.
\end{equation}
On the numerical analysis side, for discretized equations, almost preservation of the super-actions at high regularity and/or almost preservation of the Hamiltonian have also been proven for other semi-linear models (including nonlinear Schr\"odinger equations; see e.g. \cite{BFG13,MR2222808,GL10,MR3712186,MR2811583,MR2895408,MR2570074}).

\medskip

\noindent \emph{\underline{Comments about the regularity condition.}} In order to control the variation of the super-actions for very long times, we have to be able to prove that the norm of the solution remains of order $\varepsilon$ (or at least $\varepsilon^{\alpha}$ with $\alpha >1/2$) for very long times. In high regularity (i.e. $s\gg 1$), this is usually done by a bootstrap argument because the estimate \eqref{eq:classicul} on the variation of the super-actions implies the almost preservation of the $H^{s+1/2}\times H^{s-1/2}$ norm of the solution, i.e. $\|  (q(t),p(t)) \|_{H^{s+1/2}\times H^{s-1/2}} $ remains of size $\varepsilon$ for times of order $\varepsilon^{-r}$ (see e.g. \cite{BG06,MR2366141}).

Unfortunately, in low regularity, we only have a very weak control of the high super-actions (due to the factor $\beta_r$ in \eqref{eq:low_reg_poly}). The estimate \eqref{eq:low_reg_poly}, only provides a control of a Sobolev norm of index much smaller than the one of the norm in which we control the size of the initial data. In other words, it is an estimate with a loss of derivatives. It is too weak for a bootstrap. As a consequence, we need an \emph{a priori estimate} on the norm of the solution. That is why we consider small initial data in the energy space $H^1\times L^2$ in order to exploit the preservation of the Hamiltonian to ensure that the solution remains small in this space. 

We could also use other conserved quantities (like the Gibbs measure, or the mass for nonlinear Schr\"odinger equations; see \cite{bernier:hal-04090717}) to get a priori estimates. Unfortunately, in general, there exists  very few conserved quantities and they only control low regularity Sobolev norms (which, as explained in \cite{BGR21,bernier:hal-04090717}, generates many technical obstructions).

\subsection{Discussions about the proof}

\label{sub:proof}

The proof of Theorem \ref{thm:main} is divided into three main steps. First, in Section \ref{sec:back}, we perform the \emph{backward error analysis} of the method. More precisely, inspired by \cite{MR2895408,MR2811583,BFG13}, we construct, for any $r\geq 1$ arbitrary large but given, a modified Hamiltonian $H_h$ associated to a modified semi-discretized equation whose flow at time $h$ is an approximation of order $\varepsilon^{r+2}$ of the numerical flow (see Theorem \ref{main}). It is at this step that, to avoid resonant time steps, we impose the CFL condition \eqref{eq:CFL}. At this step, the regularity plays no role, we consider initial data of size $\varepsilon$ in $H^{s+1/2}\times H^{s-1/2}$ with $s\geq 1/2$ (but $s>0$ suffices).

As a corollary, we deduce the almost preservation of the modified Hamiltonian for times of order $\varepsilon^{-r}$ (see Proposition \ref{conservation} and Corollary \ref{cor:reste_petit}) and so that the numerical solution $(q^n,p^n)$ remains of size $\varepsilon$ in $H^1\times L^2$ on this time scale (see Corollary \ref{cor:reste_petit}).

Then, in Section \ref{z3}, we present an abstract theorem allowing to put the Hamiltonian $H_h$ in (partial) Birkhoff normal form. In other words,  we perform a canonical change of variable, close to the identity, to remove all the terms of $H_h$ which do not commute (or almost do not commute) with the linear part of the equation (given by $T$).

Finally, in Section \ref{z4}, we use the non-resonance conditions of \cite{bernier2021birkhoff} and the change of variable given by the Birkhoff normal form to construct some modified super-actions which are almost preserved by the flow generated by $H_h$ and so, thanks to the backward error analysis, by the numerical flow. 

It seems to us that this approach is robust and could be applied to other models and numerical methods. The main limitation comes from the small divisor estimates which impose that the eigenvalues of the linear part converge to the integers (see \eqref{eq:main_lim}).


The approach developed in this paper could also provide an alternative proof of the result in high regularity of \cite{zbMATH05323311} (only the dynamical consequence of the Birkhoff normal form theorem in Section \ref{z4} should be replaced by the classical one in high regularity (as in \cite{Bam03} or \cite{BG06} for example)).

\subsection{Numerical experiments} \label{sub:exp} To highlight the results mentioned in this paper, we now present some numerical experiments.

\medskip

For the experiments, for each regularity exponent $s_0 \in \{ 0.5,1\}$, we choose
$$
\forall y \in \mathbb{R}, \ g(y) = -y^5,\, K = 2048,  \, \rho = \sqrt{8}, \, \Phi = 1, h=10^{-3},\, \varepsilon = 0.75,
$$
$$
\forall x\in \mathbb{T}_K, \ q^0(x) = \varepsilon Z^{-1} \sum_{k\in \mathcal{N}_K} \langle k \rangle^{-s_0 - 0.525} e^{ikx} \quad and \quad p^0(x)=0.
$$
where 
$$
Z := \big(\sum_{k\in \mathcal{N}_K} \langle k \rangle^{ - 1.05  }\big)^{1/2}
$$
 is a normalization constant chosen such that 
 $$
 \|q^0\|_{H^{s_0}} = \varepsilon.
 $$
We note that this corresponds to the discretization of an initial datum $q(0)$ which is in $H^{s}$ for any $s<s_0+0.05$.
Then, for all $n\geq 1$, we compute $(q^n,p^n)$ using \eqref{eq:KG_fd}. 

\medskip

In Figure \ref{fig:1}, we plot the evolution of the super actions $J_k(q^n,p^n)$ for $hn\leq T=4000$, $s_0 \in \{0.5,1\}$ and $k \in \{0,\cdots,14\} \cup \{59,\cdots,89\}$. We observe that, in any case, the super actions are almost preserved. The larger $k$ and $s_0$ are, the lower the amplitude of the variation is.

\medskip

Sub-figure $(B)$ corresponds qualitatively\footnote{i.e. up to some normalization constants.} to what we proved in Theorem \ref{thm:main}: low super-actions are almost preserved in the energy space. We have no explanation for the almost conservation of the high super-actions (Sub-figures $(C)$ and $(D)$). If we knew it, we could adapt the proof of Theorem \ref{thm:main} to deduce the almost preservation of the low super-actions below the energy space (i.e. Sub-figure $(A)$).

\medskip

In the case $s_0= 0.5$,  since $K$ is finite, $ \|q^0\|_{H^{1}}$ is also finite of course. However, it is not small: $\|q^0\|_{H^{1}} \approx 7.5055$.
From what we observed, different values of parameters ($\rho,K,\varepsilon$...), of initial data ($p^0, q^0$)  and nonlinearity $g$ provide similar numerical results. We refer to \cite{zbMATH05323311,MR2366141,MR2895408} for further numerical experiments.

\begin{figure}[htbp]

\centering
    \begin{subfigure}[b]{0.5\textwidth}
        \centering
        \includegraphics[scale=0.45]{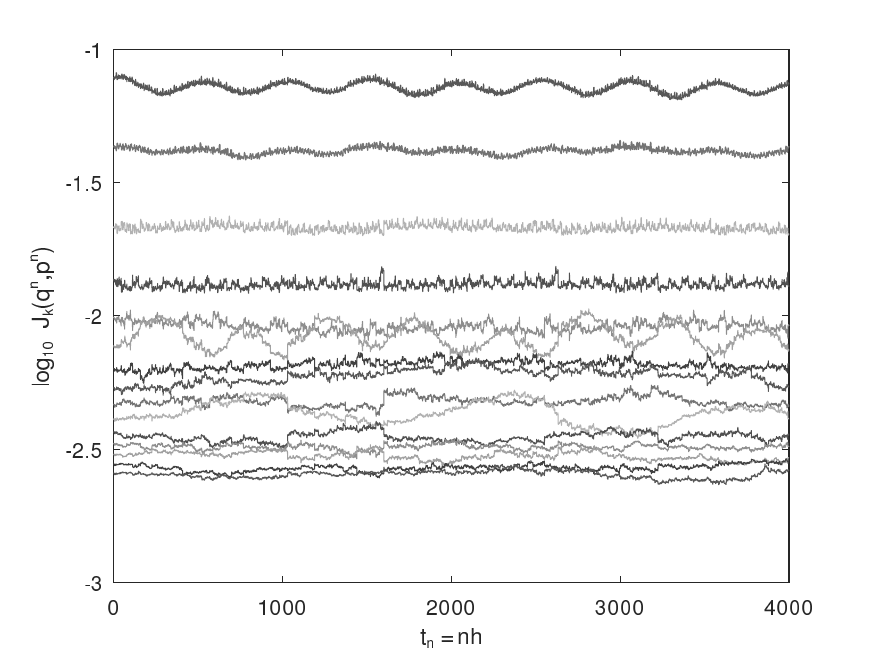}
        \caption{$s_0=0.5$, $k\in \{0,\cdots,14\}$}
        \label{sub:A}
    \end{subfigure}%
    \begin{subfigure}[b]{0.5\textwidth}
        \centering
        \includegraphics[scale=0.45]{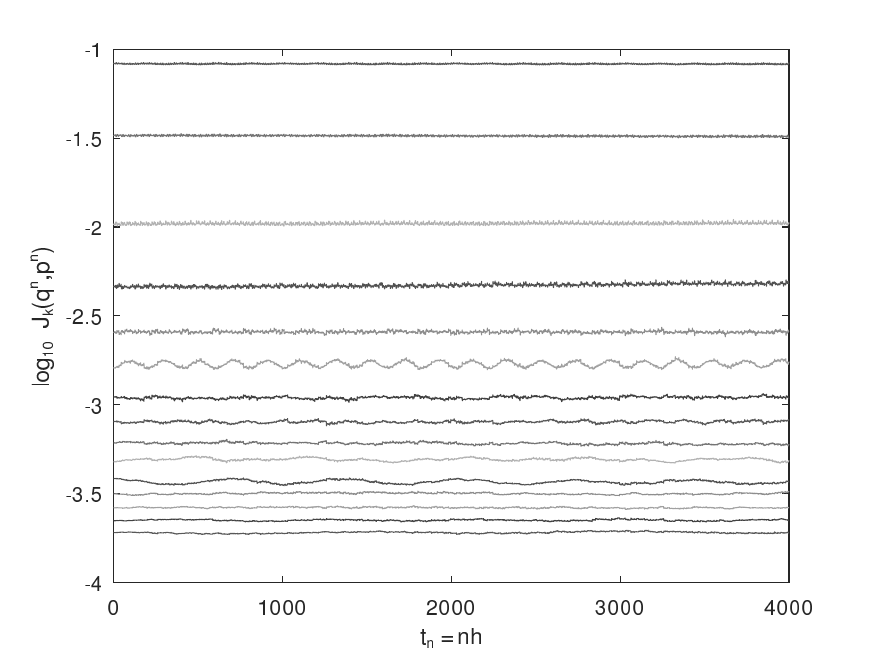}
        \caption{$s_0=1$, $k\in \{0,\cdots,14\}$}
        \label{sub:B}
    \end{subfigure}
    \begin{subfigure}[b]{0.5\textwidth}
        \centering
        \includegraphics[scale=0.45]{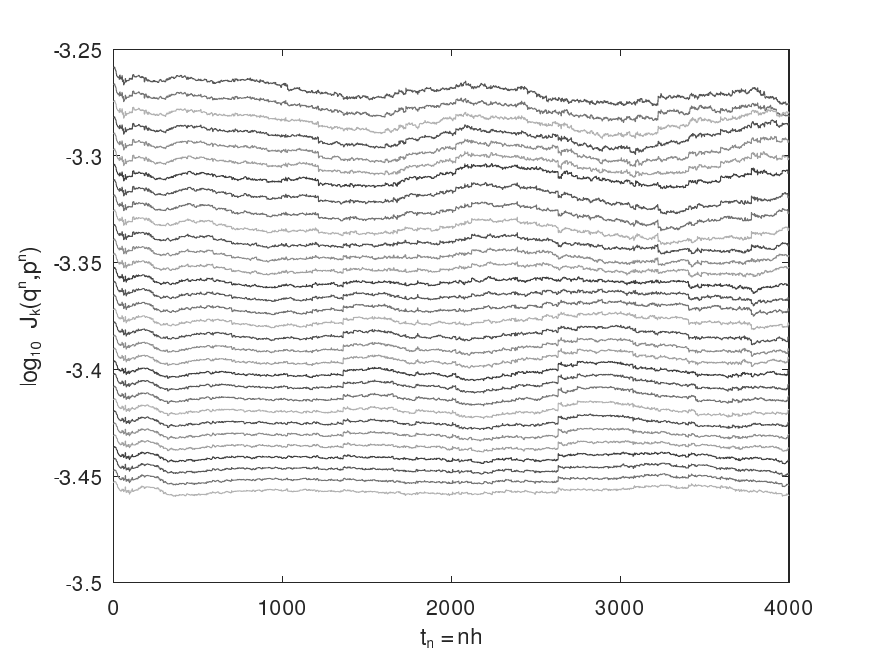}
        \caption{$s_0=0.5$, $k\in \{59,\cdots,89\}$}
    \end{subfigure}%
    \begin{subfigure}[b]{0.5\textwidth}
        \centering
        \includegraphics[scale=0.45]{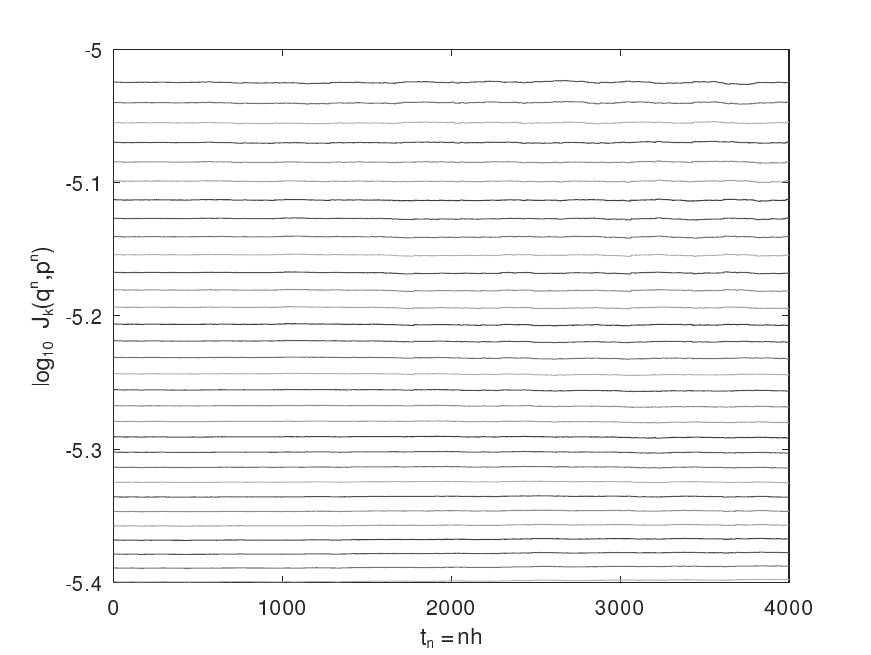}
        \caption{$s_0=1$, $k\in \{59,\cdots,89\}$}
    \end{subfigure}

\caption{Evolution of the logarithm of the super actions $\log_{10} J_k(q^n,p^n)$ for different regularity exponents $s_0$ and index ranges $k$ and for times $hn\leq T=4000$. }

\label{fig:1}
\end{figure}

\subsection*{Acknowledgments}  The authors thank E. Faou and B. Gr\'ebert for enthusiastic discussions about this work. During the preparation of this work the authors benefited from the support of the Centre Henri Lebesgue ANR-11-LABX-0020-0 and were partially supported by the ANR project KEN ANR-22-CE40-0016. J.B. was also supported by the region "Pays de la Loire" through the project "MasCan".

\section{Hamiltonian formalism}\label{z2}

In this section, we introduce a formalism well suited to deal with Hamiltonian PDEs and their discretizations. First, we introduce complex variables in order to diagonalize the linear part of \eqref{eq:KG_sd}. Then we introduce spaces of polynomial Hamiltonians which allow to perform the backward error analysis and the Birkhoff normal form. Finally, we will establish estimates on the vector fields and the flows generated by these polynomial Hamiltonians.  Since the pioneer work of Bambusi \cite{Bam03}, this kind of formalism is quite standard to perform Birkhoff normal form techniques to Hamiltonian PDEs. In the context of numerical analysis, similar formalism have, for example, also been used in \cite{MR2895408,MR2811583,MR2570074}. The main difference with respect to these papers is that, since we consider non-smooth solutions, we carefully keep track (as in \cite{bernier2021birkhoff,BGR21}) of the smoothing properties of the nonlinearity (encoded by the factors $\langle j_i \rangle^{1/2}$ in Definition \ref{ikea11} below).

\subsection{Complex variables and associated functional setting}\label{W1}

In the introduction, we have presented the problem and the result in the natural variables $(q,p)$. Nevertheless, it is much more convenient for the analysis to work in the complex variables $u\in \mathbb{C}^{\mathcal{N}_K}$. They are the variables diagonalizing the linear part of \eqref{eq:KG_sd} and are defined by
$$
u := \Lambda^{1/2}q + i \Lambda^{-1/2}p.
$$

In this subsection, we reformulate the equation and rewrite the main quantities with respect to these variables.

\noindent \underline{\emph{The harmonic actions.}} In these variables, the harmonic actions rewrite
\begin{equation}
\label{eq:les_supers_actions}
J_k(u) =  \frac{|u_k|^2 + |u_{-k}|^2}2 \quad \mathrm{if} \quad |k|<K/2
\end{equation}
and $J_{-K/2} = |u_{-K/2}|^2$ if $K$ is even. Indeed, if $|k|<K/2$, since $q,p$ are real-valued, we have
$$
p_k = \overline{p_{-k}},  \quad q_k = \overline{q_{-k}}
$$
and $p_{-K/2},q_{-K/2} \in \mathbb{R}$ if $K$ is even. It follows that if $|k|<K/2$, 
$$
\overline{u_{-k}} = \omega_k^{1/2} q_{k} - i  \omega_k^{-1/2} p_{k},
$$
and so
$$
q_k = \frac{u_k + \overline{u_{-k}}}{2 \omega_k^{1/2}} \quad \mathrm{and} \quad  p_k =\frac{u_k - \overline{u_{-k}}}{2 i\omega_k^{-1/2}}.
$$
As a consequence, recalling that by definition $J_k(p,q) :=  \omega_k^{-1} |p_k|^2 + \omega_k |q_k|^2$, we get \eqref{eq:les_supers_actions}  just by applying the parallelogram law.

\noindent \underline{\emph{The Sobolev norms.}} We note that
$$
\| q,p \|_{H^{s+1/2}\times H^{s-1/2}}^2 \sim_{\rho} \sum_{k\in \mathcal{N}_K} \langle k \rangle^{2s} J_k(q,p) \sim \| u\|_{H^s}^2.
$$
As a consequence, in Theorem \ref{thm:main}, we are going to consider solutions of size $\varepsilon$ in $H^{1/2}$.

\noindent \underline{\emph{The semi-discretized equation.}} The semi discretized nonlinear Klein--Gordon equation rewrites
$$
i\partial_t u = \Lambda u + \Lambda^{-1/2} g( \Lambda^{-1/2} \Re u) = \nabla H(u)
$$
with $H = T + W$ where the quadratic part of the Hamiltonian $T$ and its nonlinear part $W$ (defined respectively by \eqref{eq:def_TK} and \eqref{eq:def_WK}) rewrite
\begin{equation}
\label{eq:def_T_intro}
T(u) = \frac12 \sum_{k\in \mathcal{N}_K} \omega_k |u_k|^2 \quad \mathrm{and} \quad W(u)= \frac1K \sum_{x\in \mathbb{T}_K} G(\Re u(x)).
\end{equation}
The gradient $\nabla$ is the one associated with the real scalar product $L^2$ on $\mathbb{C}^{\mathcal{N}_K}$ defined in \eqref{eq:isom}. It is convenient to note that, as usual, for any smooth real-valued function $P$ on $\mathbb{C}^{\mathcal{N}_K}$, we have
$$
\nabla P(u) = 2\partial_{\overline{u_k}} P(u) \quad \mathrm{where} \quad \partial_{\overline{u_k}} := \partial_{\Re u_k} + i\partial_{\Im u_k}.
$$

\noindent \underline{\emph{The fully discretized equation.}} The fully discretized nonlinear Klein--Gordon equation rewrites naturally
\begin{equation}
\label{eq:def_Phi_num}
u^{n+1} = \Phi_{\mathrm{num}}^h(u^n)
\end{equation}
where, as previously, $\Phi_{\mathrm{num}}^h := \Phi^{h/2}_{V}\circ \Phi^{h}_{T} \circ \Phi^{h/2}_{V}$ and $\Phi_{V}, \Phi_{T}$ are the flows of the equations
$$
i\partial_t u = \nabla T (u) =\Lambda u \quad \mathrm{and} \quad i\partial_t u = \nabla V (u) = \Lambda^{-1/2} \phi(h\Lambda) g( \Lambda^{-1/2} \phi(h\Lambda) \Re u)
$$
with 
\begin{equation}
\label{eq:def_V}
V(u) = W(\Lambda^{-1/2}\phi(h\Lambda) \Re u).
\end{equation}


\noindent \underline{\emph{Poisson bracket and symplectic maps.}} We equip naturally $\mathbb{C}^{\mathcal{N}_K}$ of the symplectic form $(i\cdot,\cdot)_{L^2}$ where the real scalar product $L^2$ is given by \eqref{eq:isom}. It provides canonical notions of symplectic map and of Poisson bracket, the definitions of which we recall.

\begin{definition}[Poisson bracket]\label{ep1}Given two smooth functions $P,\chi {\;}: \,  \mathbb{C}^{\mathcal{N}_K} \to \mathbb{R}$, their \emph{Poisson bracket} is defined by
$$
\{P,\chi\}(u):={( i\nabla P(u),\nabla \chi(u))}_{L^2}.
$$
\end{definition}
\begin{remark}
As usual, it satisfies the identity
\begin{equation}
\label{eq:nice_formula}
\{P,\chi\}(u) = 2i\sum_{k\in\mathcal{N}_K} \left(\partial_{\overline{u_k}}P(u) \partial_{u_k}\chi(u) - \partial_{u_k}P(u) \partial_{\overline{u_k}}\chi(u)\right).
\end{equation}
where $\partial_{u_k} := \partial_{\Re u_k} - i\partial_{\Im u_k}$.
\end{remark}

\begin{definition}[Adjoint representation $\mathrm{ad}$]\label{def:ad}Being given two smooth functions $P,\chi {\;}: \,  \mathbb{C}^{\mathcal{N}_K} \to \mathbb{R}$, we set
$$
\mathrm{ad}_\chi P := \{\chi,P\}.
$$
\end{definition}

\begin{definition}[Symplectic Map]\label{defff4}
Consider an open set $\mathcal{C}$ of $\mathbb{C}^{\mathcal{N}_K}$ and a $C^1$ map $\tau {\;}: \,  \mathcal{C} \to \mathbb{C}^{\mathcal{N}_K}.$ We say that $\tau$ is a symplectic map if
 \[ \forall u\in \mathcal{C}, \forall v,w \in \mathbb{C}^{\mathcal{N}_K}, \hspace{0.2cm} {( iv,w)} _{L^2} = {( i \mathrm{d}\tau (u)(v), \mathrm{d}\tau (u)(w))} _{L^2}.\]
\end{definition}

\subsection{Class of Hamiltonian functions}\label{2001}

To perform the backward error analysis and to put the system in Birkhoff normal form, we are going to transform many Hamiltonians and, in particular, their Taylor expansions. To be able to prove sharp estimates on these Taylor expansions, we define some spaces of polynomials on $\mathbb{C}^{\mathcal{N}_K}$, and we describe their main properties.

\begin{definition}[Momentum $\mathcal{M}(j,\sigma)$]Let $r\in \mathbb{N}$ be given. For a collection of indices $j=(j_1,\cdots,j_{r+2})\in \mathcal{N}_K^{r+2}$ and of signs $\sigma \in \{-1,1\}^{r+2},$ we define the momentum $\mathcal{M}(j,\sigma)$ by the following formula 
$$
\mathcal{M}(j,\sigma) = \sum_{i=1}^{r+2}\sigma_ij_i.
$$

\end{definition}

\begin{definition}(Class $\mathscr{H}^{r+2})$\label{defi1}
Given $r\geq 0$, we denote by $\mathscr{H}^{r+2}$ the set of real-valued homogeneous polynomials of degree $r+2$, defined on $\mathbb{C}^{\mathcal{N}_K}$, of the form
\[ P(u)=\sum_{\substack{ j \in \mathcal{N}^{r+2}_K \\ \sigma \in \{-1,1\}^{r+2} }} P_j^{\sigma}\hspace{0.05cm} u_{j_1}^{\sigma_1}\cdots u_{j_{r+2}}^{\sigma_{r+2}} \] 
where $u_k^1:=u_k, u_k^{-1}:= \overline{u_k}$ and $(P_j^{\sigma})_{j \in \mathcal{N}_K^{r+2}} $ is a sequence of complex numbers satisfying:
\begin{itemize}
 \item the \textit{reality condition} $$P_j^{-\sigma}=\overline{P_{j}^{\sigma}}$$
 \item the \emph{momentum condition}\footnote{ $\mathcal{M}(j,\sigma) \equiv 0(K)$ means that $\mathcal{M}(j,\sigma) \in K\mathbb{Z}$.} 
 \begin{equation}
 \label{eq:zero_mom}
 P_j^{\sigma} \neq 0 \implies \mathcal{M}(j,\sigma) \equiv 0(K).
 \end{equation}
    \item the \textit{symmetry condition}  $$\forall \phi \in \mathscr{S}_{r+2}, \hspace{0.3cm} P_{j_1,\cdots , j_{r+2}}^{\sigma_1,\cdots,\sigma_{r+2}} =  P_{j_{\phi_1},\cdots , j_{\phi_{r+2}}}^{\sigma_{\phi_1},\cdots,\sigma_{\phi_{r+2}}},$$
    $\mathscr{S}_{r+2}$ denoting the group of the permutations of $\{1,\cdots,r+2\}$.
\end{itemize}
\end{definition}
Note that thanks to the symmetry condition, the coefficients $P_j^{\sigma}$ are uniquely determined by the polynomial function $P$.

We endow this space of polynomials with the following norm ${\Vert\cdot\Vert}_{\mathscr{H}}$. 
\begin{definition}(Norm ${\Vert \cdot \Vert}_{\mathscr{H}}$) \label{ikea11}
Let $r\geq 0$ and $P \in \mathscr{H}^{r+2}.$ We introduce the norm
\[{\Vert P \Vert}_{\mathscr{H}}:= \sup_{\substack{j \in \mathcal{N}^{r+2}_K\\ \sigma\in \{-1,1\}^{r+2}}} \abs{P_j^{\sigma}} {\langle j_1 \rangle}^{1/2}\cdots {\langle j_{r+2} \rangle}^{1/2} .\] 
\end{definition}

Important examples of polynomials in these spaces are those of the Taylor expansion of $V \in \mathcal{C}^\infty( \mathbb{C}^{\mathcal{N}_K}; \mathbb{R})$,  defined by \eqref{eq:def_V}, in $u=0$.
\begin{lemma} \label{lem:def_Taylor_P} For all $n\geq 1$, we have that 
\begin{equation}
\label{eq:def_Pn}
P_n :=u\mapsto  \frac1{(n+2)!}\mathrm{d}^{n+2} V(0)(u,\cdots,u) \in \mathscr{H}^{n+2}
\end{equation}
is uniformly bounded with respect to $h$ and $K$, i.e.
$$
\| P_n\|_{\mathscr{H}} \lesssim_{n,\rho} 1.
$$
\end{lemma}
\begin{proof}
By definition of $V$, we have that
$$
P_n(u) = \frac{c_n}K \sum_{x\in \mathbb{T}_K}(v(x))^{n+2} \quad \mathrm{where} \quad v=\Lambda^{-1/2}\phi(h\Lambda) \Re u
$$
and $c_n=((n+2)!)^{-1}G^{(n+2)}(0).$ Thanks to the Fourier inversion formula, it follows that
$$
P_n(u) = c_n \sum_{j_1+\cdots+j_{n+2} \equiv 0(K)} v_{j_1} \cdots v_{j_{n+2}}.
$$
Finally using that by definition for $j\in \mathcal{N}_K$
$$
v_j = \omega_{j}^{-1/2} \phi(h \omega_j) \frac{u_j + \overline{u}_j}2 = 2^{-1} \omega_{j}^{-1/2} \phi(h \omega_j) (u_j + \overline{u_{-j}}),
$$
(where $u_{K/2} = u_{-K/2}$ if $K$ is even), we get
$$
P_n(u) = 2^{-n-2} c_n \sum_{\sigma \in \{-1,1\}^{n+2}} \sum_{\mathcal{M}(j,\sigma) \equiv 0(K)} \omega_{j_1}^{-1/2} \phi(h \omega_{j_1})  u_{j_1}^{\sigma_1} \cdots \omega_{j_{n+2}}^{-1/2} \phi(h \omega_{j_{n+2}}) u_{j_{n+2}}^{\sigma_{n+2}},
$$
i.e.
$$
(P_n)_j^\sigma = 2^{-n-2} c_n\mathds{1}_{\mathcal{M}(j,\sigma) \equiv 0(K)} \omega_{j_1}^{-1/2} \phi(h \omega_{j_1}) \cdots \omega_{j_{n+2}}^{-1/2} \phi(h \omega_{j_{n+2}}).
$$
\end{proof}

\subsection{Vector field and Poisson bracket estimates}

Before proving vector field estimates for the Hamiltonians in $\mathscr{H}^{r+2}$, we give a technical lemma useful to deal with aliasing error terms.
\begin{lemma}\label{aliasing}
Let $r\geq 0.$ For all $m\in \mathbb{Z}^*, j\in \mathcal{N}_K^{r+1}, \sigma\in \{-1,1\}^{r+1}$ and $k\in \mathcal{N}_K$ satisfying $j_1 \leq j_2\leq \cdots \leq j_{r+1}$ and $mK= \sigma_1j_1+\cdots + \sigma_{r+1}j_{r+1} -k$, we have that \[\langle K\rangle \leq 2(r+1)\langle j_{r+1}\rangle.\]
\end{lemma}

\begin{proof}
Since we assume that $j_1 \leq j_2\leq \cdots \leq j_{r+1},$ then we have
\[\abs{m}\abs{K} = \abs{mK} = \abs{\sigma_1j_1+\cdots + \sigma_{r+1}j_{r+1} -k} \leq \abs{k} + (r+1)\abs{j_{r+1}} \leq \frac{\abs{K}}{2} +  (r+1)\abs{j_{r+1}}\]
and so, since $|m|\geq 1$, as expected,
$$
\dfrac{\abs{K}}{2} \leq (\abs{m} - \frac12)\abs{K} \leq (r+1)\abs{j_{r+1}}.
$$ 
\end{proof}

Now we turn to the estimate on the gradient provided by the $\mathscr{H}$-norm.
\begin{proposition}\label{vectorfield}
There exists $C>1$ such that for all $r \geq 0, s\geq 1/2$ and $P\in \mathscr{H}^{r+2},$ the gradient of $P$ is a smooth function from $\mathbb{C}^{\mathcal{N}_K}$ to $\mathbb{C}^{\mathcal{N}_K}$ enjoying the bound
  \[\forall u \in \mathbb{C}^{\mathcal{N}_K}, \hspace{0.3cm} {\Vert\nabla P(u)\Vert}_{H^s}\lesssim_s C^r {\Vert P \Vert}_{\mathscr{H}} {\Vert  u\Vert}_{H^s}^{r+1}.\]
\end{proposition}

\begin{proof} 
By the definitions of the gradient, the $\mathscr{H}-$norm and the symmetry condition, we get 
\begin{equation*}
\begin{split}
 {\Vert\nabla P(u)\Vert}_{H^s}^2 &\leq 4(r+2)^2 \sum_{k\in  \mathcal{N}_K}\langle k \rangle^{2s}\Bigg[ \sum_{\substack{j \in \mathcal{N}^{r+1}_K\\ \sigma\in \{-1,1\}^{r+1} \\ \mathcal{M}(j,k,\sigma,-1) \equiv 0(K)}} \abs{P_{j,k}^{\sigma,-1}}\abs{u_{j_1}^{\sigma_1}}\cdots\abs{u_{j_{r+1}}^{\sigma_{r+1}}}  \Bigg]^2\\
&\leq 4(r+2)^2 {\Vert P \Vert}_{\mathscr{H}}^2 \sum_{k\in  \mathcal{N}_K} \Bigg[\sum_{\substack{j \in \mathcal{N}^{r+1}_K\\ \sigma\in \{-1,1\}^{r+1} \\ \mathcal{M}(j,k,\sigma,-1) \equiv 0(K)}}  {\langle k \rangle}^{s} \prod_{i=1}^{r+1}{\langle j_i \rangle}^{-1/2} \abs{u_{j_i}}  \Bigg]^2.
\end{split}
\end{equation*}
Moreover, if $j\in \mathcal{N}^{r+1}_K, \sigma \in \{-1,1\}^{r+1}$ satisfy $\mathcal{M}(j,k,\sigma,-1) \equiv 0(K)$, then there exists $m \in \mathbb{Z}$ such that 
$$
\sigma_1 j_1 + \cdots \sigma_{r+1} j_{r+1}-k = m K.
$$
Now, since $\abs{m}\leq r/2$ and applying Lemma {\ref{aliasing}}, we have 
\begin{align*}
\langle k \rangle^s&= \langle \sigma_1j_1+\cdots + \sigma_{r+1}j_{r+1} - mK\rangle^s \\
&\leq \left[(r+1)\max_{n=1,\cdots, r+1}\langle j_{n} \rangle +\langle mK \rangle\right]^s \\
&\leq 2^{(s-1)_+}\left[(r+1)^s\max_{n=1,\cdots, r+1}\langle j_{n} \rangle^s + \frac{r^s}{2^s}\mathds{1}_{m\neq 0} \langle K \rangle^s\right]\\
&\leq 2^{(s-1)_+}\left[(r+1)^s\max_{n=1,\cdots, r+1}\langle j_{n} \rangle^s + r^s (r+1)^s \max_{n=1,\cdots, r+1}\langle j_{n} \rangle^s\right]\\
&\leq 2^{(s-1)_+}(r+1)^s(1 +r^s)\max_{n=1,\cdots, r+1}\langle j_{n} \rangle^s
\end{align*}
where $(s-1)_+ := \max{(0,s-1)}.$ Then, replacing back, we get that
$$
{\Vert\nabla P(u)\Vert}_{H^s}^2 \lesssim_s c^r {\Vert P \Vert}_{\mathscr{H}}^2 \sum_{k\in  \mathcal{N}_K} \Bigg[\sum_{\substack{j \in \mathcal{N}^{r+1}_K\\ \sigma\in \{-1,1\}^{r+1} \\ \mathcal{M}(j,k,\sigma,-1) \equiv 0(K)}} \max_{n=1,\cdots, r+1}\langle j_{n} \rangle^s \prod_{i=1}^{r+1}{\langle j_i \rangle}^{-1/2} \abs{u_{j_i}}  \Bigg]^2
$$
for some constant $c>1$. Notice that 
\begin{align*}
\sum_{\substack{j \in \mathcal{N}^{r+1}_K\\ \sigma\in \{-1,1\}^{r+1} \\ \mathcal{M}(j,k,\sigma,-1) \equiv 0(K)}} \! \! \! \max_{n=1,\cdots, r+1}\langle j_{n} \rangle^s \prod_{i=1}^{r+1}{\langle j_i \rangle}^{-1/2} \abs{u_{j_i}} &\leq \sum_{\substack{j \in \mathcal{N}^{r+1}_K\\ \sigma\in \{-1,1\}^{r+1} \\ \mathcal{M}(j,k,\sigma,-1) \equiv 0(K)}}{\langle j_1 \rangle}^{s-1/2} \abs{u_{j_1}}
 \prod_{i=2}^{r+1}{\langle j_i \rangle}^{-1/2} \abs{u_{j_i}} + \cdots\\
 &+ \sum_{\substack{j \in \mathcal{N}^{r+1}_K\\ \sigma\in \{-1,1\}^{r+1} \\ \mathcal{M}(j,k,\sigma,-1) \equiv 0(K)}}{\langle j_{r+1} \rangle}^{s-1/2} \abs{u_{j_{r+1}}} \prod_{i=1}^{r}{\langle j_i \rangle}^{-1/2} \abs{u_{j_i}}.
\end{align*}
Then, after re-indexing, there exists a constant $C>1$ such that
$$
{\Vert\nabla P(u)\Vert}_{H^s}^2 \lesssim_s C^r {\Vert P \Vert}_{\mathscr{H}}^2 \sum_{k\in  \mathcal{N}_K} \Bigg[\sum_{\substack{j \in \mathcal{N}^{r+1}_K\\ \sigma\in \{-1,1\}^{r+1} \\ \mathcal{M}(j,k,\sigma,-1) \equiv 0(K)}}{\langle j_{r+1} \rangle}^{s-1/2} \abs{u_{j_{r+1}}} \prod_{i=1}^{r}{\langle j_i \rangle}^{-1/2} \abs{u_{j_i}}  \Bigg]^2.
$$
Using Young's convolution inequality $\ell^2 * \ell^1*\ell^1* \cdots *\ell^1 \hookrightarrow \ell^{2}$, we have
\begin{align*}
 {\Vert\nabla P(u)\Vert}_{H^s}^2&\lesssim_s  C^r {\Vert P \Vert}_{\mathscr{H}}^2 {\Vert \langle \cdot \rangle^{s-1/2} u\Vert}_{\ell^2}^2 {\Vert \langle \cdot \rangle^{-1/2} u\Vert}_{\ell^1}^{2r}\lesssim_s C^r {\Vert P \Vert}_{\mathscr{H}}^2 {\Vert  u\Vert}_{H^s}^2  {\Vert \langle \cdot \rangle^{-1/2} u\Vert}_{\ell^1}^{2r}.
\end{align*}
Then, by Cauchy--Schwarz and since $s\geq 1/2,$ we deduce the desired result.
\end{proof}
Since we deal with polynomials, we deduce by multi-linearity the following estimate (note that this is also a quite direct corollary of the proof of Proposition \ref{vectorfield}).
\begin{corollary}\label{difference}
There exists $C>1$ such that for all $r \geq 0, s\geq 1/2$ and $P\in \mathscr{H}^{r+2},$ we have \[\forall u \in \mathbb{C}^{\mathcal{N}_K}, \hspace{0.3cm} {\Vert\mathrm{d}\nabla P(u) \Vert}_{H^s}\lesssim_s C^{r} {\Vert P \Vert}_{\mathscr{H}} {\Vert  u\Vert}_{H^s}^{r}.\]
\end{corollary}

Since $(\nabla P(u),u)_{L^2} = (r+2) P(u)$, we also deduce the following estimate as corollary.
\begin{corollary}\label{cor:ev_pol}
There exists $C>1$ such that for all $r \geq 0$ and $P\in \mathscr{H}^{r+2},$  $P$ is a smooth function from $\mathbb{C}^{\mathcal{N}_K}$ to $\mathbb{R}$ enjoying the bound
  \[\forall u \in \mathbb{C}^{\mathcal{N}_K}, \hspace{0.3cm} | P(u)| \lesssim C^r {\Vert P \Vert}_{\mathscr{H}} {\Vert  u\Vert}_{H^{1/2}}^{r+2}.\]
\end{corollary}

We shall prove after this that the spaces of Hamiltonians are stable by the Poisson brackets.
\begin{proposition}\label{1209}
Let $P \in \mathscr{H}^{r+2}$ and $\chi \in \mathscr{H}^{r'+2}$ with $r,r' \geq 0.$ Then, there exists a Hamiltonian $N \in \mathscr{H}^{r+r'+2}$ such that 
$$\forall u \in \mathbb{C}^{\mathcal{N}_K},\hspace{0.3cm} \{P,\chi\}(u) = N(u)$$ and $${\Vert\{P,\chi\}\Vert}_{\mathscr{H}} \leq 4 (r+2)(r'+2) {\Vert P\Vert}_{\mathscr{H}}{\Vert\chi\Vert}_{\mathscr{H}}.$$
\end{proposition}

\begin{proof}
Let $u\in \mathbb{C}^{\mathcal{N}_K}.$ We express the Hamiltonians as
\begin{align*}
P(u)&=\sum_{\substack{ j \in \mathcal{N}^{r+2}_K \\ \sigma \in \{-1,1\}^{r+2} \\ \mathcal{M}(j,\sigma) \equiv 0(K)}} P_j^{\sigma}\hspace{0.05cm} u_{j_1}^{\sigma_1}\cdots u_{j_{r+2}}^{\sigma_{r+2}} \quad \text{and}\quad \chi(u)=\sum_{\substack{ j' \in \mathcal{N}^{r'+2}_K \\ \sigma ' \in \{-1,1\}^{r'+2} \\ \mathcal{M}(j',\sigma ') \equiv 0(K)}} \chi_{j'}^{\sigma '}\hspace{0.05cm} u_{j_1'}^{\sigma_1'}\cdots u_{j_{r'+2}'}^{\sigma_{r'+2}'}.\end{align*}
By Definition \ref{ep1} and the symmetry condition satisfied by the coefficients of $P$ and $\chi$
\begin{align*} 
 & \{P,\chi\}(u)\\
 &= 2i\sum_{k \in \mathcal{N}_K} (\partial_{\overline{u_k}}P(u) \partial_{u_k}\chi(u) - \partial_{u_k}P(u) \partial_{\overline{u _k}}\chi(u)) \\
&= \sum_{\substack{ (j,j') \in \mathcal{N}^{r+r'+2}_K \\ (\sigma,\sigma') \in \{-1,1\}^{r+r'+2} }} 2i(r+2)(r'+2) \sum\limits_{k\in \mathcal{N}_K}\left( P_{j,k}^{\sigma,-1} \chi_{j',k}^{\sigma',1} - P_{j,k}^{\sigma,1} \chi_{j',k}^{\sigma',-1}\right) u_{j_1}^{\sigma_1}\cdots u_{j_{r+1}}^{\sigma_{r+1}} u_{j_1'}^{\sigma_1'}\cdots u_{j_{r'+1}'}^{\sigma_{r'+1}'}.
\end{align*}
We set 
\[M_{j''}^{\sigma''}:=  2i(r+2)(r'+2) \sum\limits_{k\in \mathcal{N}_K}\left( P_{j,k}^{\sigma,-1} \chi_{j',k}^{\sigma',1} - P_{j,k}^{\sigma,1} \chi_{j',k}^{\sigma',-1}\right)\quad\text{ and }\quad N_{j''}^{\sigma''}= \frac{1}{(r''+2)!} \sum_{\phi \in \mathscr{S}_{r''+2}}M_{j''\circ \phi}^{\sigma''\circ \phi}\]
where $j'':=(j,j')$, $\sigma'':=(\sigma,\sigma')$ and $r'':=r+r'$. Then we get
\begin{align*}
    \{P,\chi\}(u) &= \sum_{\substack{ j'' \in \mathcal{N}^{r''+2}_K \\ \sigma'' \in \{-1,1\}^{r''+2}}}N_{j''}^{\sigma''} u_{j_1''}^{\sigma_1''}\cdots u_{j_{r''+2}''}^{\sigma_{r''+2}''}=N(u).
\end{align*}
Note that we can interchange the order of summation since we are dealing with finite sums.
 Moreover, we can obviously see that $N(u)$ defines a homogeneous polynomial of degree $r''+2$ (i.e. $N \in \mathscr{H}^{r''+2}$ where both the symmetry and reality conditions of $N_{j''}^{\sigma''}$ are a direct consequence of those satisfied by $P_{j}^{\sigma}$ and $\chi_{j'}^{\sigma'}$). 
  We have to check the zero momentum condition \eqref{eq:zero_mom}. Indeed, if $M_{j''}^{\sigma''}\neq 0$ then there exists $k\in \mathcal{N}_K$ such that $P_{j,k}^{\sigma,-1} \chi_{j',k}^{\sigma',1} \neq 0$ or $P_{j,k}^{\sigma,1} \chi_{j',k}^{\sigma',-1} \neq 0$. As a consequence since $P$ and $\chi$ satisfy the zero momentum condition, there exists $\ell = \pm k$ such that
\[  j_1\sigma_1+\cdots + j_{r+1}\sigma_{r+1}-\ell    \equiv 0(K)  \quad \text{and} \quad  j_1'\sigma_1'+\cdots + j_{r'+1}'\sigma_{r'+1}'+\ell  \equiv 0(K), \]
and so
$$
 j_1\sigma_1+\cdots + j_{r+1}\sigma_{r+1} + j_1'\sigma_1'+\cdots + j_{r'+1}'\sigma_{r'+1}'  \equiv 0(K).
$$
Then, we note that thanks to the zero momentum condition \eqref{eq:zero_mom}, there exists at most one $k\in \mathcal{N}_K$ such that $P_{j,k}^{\sigma,-1} \chi_{j',k}^{\sigma',1} \neq 0$. It follows that 
 \begin{align*}
 {\Vert\{P,\chi\}\Vert}_{\mathscr{H}}&= \sup_{\substack{(j,j') \in \mathcal{N}^{r''+2}_K\\ (\sigma,\sigma')\in \{-1,1\}^{r''+2}}} \abs{M_{j''}^{\sigma''}} {\langle j_1 \rangle}^{1/2}\cdots {\langle j_{r+1} \rangle}^{1/2}  {\langle j_1' \rangle}^{1/2}\cdots {\langle j_{r'+1}' \rangle}^{1/2}\\
 &\leq 4(r+2)(r'+2) {\Vert P \Vert}_{\mathscr{H}}{\Vert \chi \Vert}_{\mathscr{H}}.
  \end{align*}
\end{proof}

\begin{lemma}\label{poissonquad}
Let $r\geq 0, s>0, P \in \mathscr{H}^{r+2}$ and $Z\in \mathscr{H}^2$ be a quadratic Hamiltonian of the form $Z(u)= \sum_{j \in \mathcal{N}_K}\lambda_j\abs{u_j}^2$ for some $\lambda \in \mathbb{R}^{\mathcal{N}_K}$. Then for all $j \in \mathcal{N}^{r+2}_K$ and all $\sigma \in \{-1,1\}^{r+2}$, we have
$$
\{P,Z\}_{j}^\sigma =-2i (\sigma_1\lambda_{j_1} + \cdots + \sigma_{r+2}\lambda_{j_{r+2}}) P_j^{\sigma}
$$
\end{lemma}

\begin{proof}
This is a direct consequence of the formula \eqref{eq:nice_formula} for the Poisson bracket and the symmetry of the coefficients of $P$.

 \end{proof}

\subsection{Flows} To perform the backward error analysis of the method and to put the modified equation in Birkhoff normal form, we are going to consider many auxiliary flows (in which the nonlinear part is always considered as perturbative). In the following proposition, we prove their existence and their basic estimates. 

\begin{proposition} \label{prop:flow} Let $r\geq 1$, $s\geq 1/2$, $C>0$, $\gamma \in (0,1)$, $\lambda \in \mathbb{R}^{\mathcal{N}_K}$, $Z : \mathbb{C}^{\mathcal{N}_K} \to \mathbb{R}$ be a quadratic Hamiltonian of the form
$$
Z(u) = \frac12 \sum_{j\in \mathcal{N}_K} \lambda_j |u_j|^2
$$
and $\chi = \chi_1+\cdots+\chi_r$ be a polynomial of degree smaller than or equal to $r+2$ such that for all $i\in \{1,\cdots,r\}$, $\chi_i\in \mathscr{H}^{i+2}$ satisfies $\| \chi_i\|_{\mathscr{H}} \leq C \gamma^{-i}$. Then there exists $\varepsilon_1 \gtrsim_{C,s,r} \gamma$ such that if $u\in \mathbb{C}^{\mathcal{N}_K}$ satisfies $\| u\|_{H^s} \leq 2\varepsilon_1$ then the flow of the equation
\begin{equation}
\label{eq:mon_equ_dev}
\left\{ \begin{array}{lll} i\partial_t v = \nabla (Z+\chi) (v) \\
v(0) = u
\end{array} \right.
\end{equation}
exists for $|t|\leq 1$ and satisfies for all $t\in [-1,1]$
\begin{equation}
\label{eq:proche_isom}
\| v(t) - \Phi_Z^t (u) \|_{H^s} \leq \left( \frac{\| u\|_{H^s}}{2\varepsilon_1} \right) \| u\|_{H^s}
\end{equation}
where the flow of $Z$ is an isometry on $H^s$ satisfying $(\Phi_Z^t u)_k = e^{-it\lambda_k} u_k$ for all $k\in \mathcal{N}_K$.
Moreover, denoting by $\Phi_{Z+\chi}$ the flow of \eqref{eq:mon_equ_dev}, we have, for all $t\in [-1,1]$, all $w\in \mathbb{C}^{\mathcal{N}_K}$ and all $u\in \mathbb{C}^{\mathcal{N}_K}$ satisfying $\| u\|_{H^s} \leq 2\varepsilon_1$
\begin{equation}
\label{eq:deriv_flow}
\| \mathrm{d}\Phi_{Z+\chi}^t (u)(w)\|_{H^s } \lesssim_{C,r,s} \|w\|_{H^s}.
\end{equation}
\end{proposition}
\begin{proof}
By applying the Duhamel formula, we have
$$
v(t) - \Phi_Z^t (u) = \int_0^t \Phi_Z^{t-\tau} \nabla \chi(v(\tau)) \mathrm{d}\tau.
$$
Since $\Phi_Z^{t-\tau}$ is an isometry, applying the vector field estimate of Proposition \ref{vectorfield}, we have
\begin{equation}
\label{eq:bientotot}
\| v(t) - \Phi_Z^t (u)\|_{H^s} \leq K_{r,s,C} |t|\max_{\tau \in [0,t]} \max_{1\leq i \leq r} \gamma^{-i} \|v(\tau) \|_{H^s}^{i+1}.
\end{equation}
where $K_{r,s,C}>1$ is a constant depending only on $r,s$ and $C$. We set $\varepsilon_1 := (9K_{r,s,C})^{-1} \gamma$.  We note that if $\| u\|_{H^s} \leq 2\varepsilon_1$ and 
$$
\sup_{|\tau|\leq |t|}\| v(\tau) - \Phi_Z^t (u) \|_{H^s} \leq \left( \frac{\| u\|_{H^s}}{2\varepsilon_1} \right) \| u\|_{H^s}
$$
then $\|v(\tau)\|_{H^s} \leq 2 \| u\|_{H^s}$ and by \eqref{eq:bientotot},
$$
\| v(t) - \Phi_Z^t (u)\|_{H^s} \leq |t| \left( \frac{\| u\|_{H^s}}{3\varepsilon_1} \right) \| u\|_{H^s} < \left( \frac{\| u\|_{H^s}}{2\varepsilon_1} \right) \| u\|_{H^s}.
$$
whenever $|t|\leq 1$.
It follows by a basic bootstrap argument that $v(t)$ exists for $|t|\leq 1$ and satisfies \eqref{eq:proche_isom}.

\medskip

It only remains to prove \eqref{eq:deriv_flow}. First, by taking the derivative of the equation and recalling that $\nabla Z$ is linear, we have
$$
i\partial_tz(t) = \mathrm{d} \nabla (Z+\chi)(v(t))(z(t) ) = \nabla Z(z(t)  ) + A_t( z(t) ). 
$$
where we have set $z(t) := \mathrm{d}\Phi_{Z+\chi}^t (u)(w)$ and $A_t =  \mathrm{d} \nabla \chi(v(t))$. Therefore, setting $y(t) = \Phi_{Z}^{-t} z(t)$, we have
$$
i\partial_t y = B_t y \quad \mathrm{where} \quad B_t = \Phi_{Z}^{-t}  A_t \Phi_{Z}^{t} y. 
$$
Now, recalling that $\Phi_{Z}^{t}$ is an isometry on $H^s$ and the bound on $A_t$ given by Corollary \ref{difference}, we have
$$
\| B_t \|_{H^s \to H^s} \lesssim_{C,r,s}  \max_{1\leq i \leq r} \gamma^{-i} \|v(\tau) \|_{H^s}^{i} \lesssim_{C,r,s} 1.
$$
Therefore, we directly conclude the proof by applying the Gr\"onwall inequality.

\end{proof}

\section{Backward error analysis} \label{sec:back}

In this section, we fix $r\geq 0$, the mass $\rho>0$, $\delta\in(0,\pi)$ and we assume that the numerical parameters $h,K$ satisfy the CFL condition \eqref{eq:CFL}, i.e.
\begin{equation*}
(r+2) \, h\, \omega_{K/2} \leq 2\pi -\delta.
\end{equation*}
 We denote by $P=P_1+\cdots+P_r$ the Taylor expansion of $V$ (defined by \eqref{eq:def_V}) up to order $r+2$ as defined by equation \eqref{eq:def_Pn} in Lemma \ref{lem:def_Taylor_P}.

\subsection{Formal derivation of the system of cohomological equations} To perform the backward error analysis of the Lie splitting, we are going to have to solve a system of cohomological equations (see \eqref{eq:the_system}). For motivational purposes, in this subsection, we present a formal derivation\footnote{we won't be using it hereafter, so we do not try to be rigorous.} of this system. It is similar to the one done in Section 3 of \cite{MR2811583}. However, there are two main differences:
\begin{itemize}
\item contrary to \cite{MR2811583}, $V$ is not a homogeneous polynomial. Therefore powers of $h$ and $u$ do not match as well as in \cite{MR2811583} (i.e., concretely, the polynomials $t\mapsto B_i(t)$ defined below are not homogeneous in $t$).
\item we do not invert formal series of non diagonal operators acting on Hamiltonians\footnote{i.e. we do not invert the operator $\varphi(\mathrm{ad}_{B(t)})$ below.}. This significantly simplifies the justification of formal calculations. However, it makes our system slightly more implicit than the one in \cite{MR2811583} but this is not a real problem because it is still triangular.
\end{itemize}

\medskip

We are looking for a family of time dependent analytic Hamiltonians $B(t,u)= B_0(u) + \sum_{n\geq 1} B_n(t,u)$ where $B_0 = hT$ and, for $n\geq 1$, $B_n(t) \equiv B_n(t,\cdot) \in \mathscr{H}^{n+2}$ is a homogeneous  polynomial of degree $n+2$  satisfying $B_n(0)= 0$ and such that, provided that $|t|$ is small enough
\begin{equation}
\label{eq:pouette}
\Phi_{B(t)}^1 =  \Phi_V^t \circ  \Phi_{T}^h.
\end{equation}
We recall that $T$, defined by \eqref{eq:def_T_intro}, denotes the quadratic Hamiltonian corresponding to the linear part of \eqref{eq:KG_sd}. Moreover, to avoid possible confusion, we point out that $\Phi_{B(t)}^1$ does not denote the flow of a non-autonomous equation.

\medskip

We note that, by definition of $B_0$, since $B(0) = B_0$, \eqref{eq:pouette} holds if $t=0$. Thus, to get \eqref{eq:pouette}, it is necessary and sufficient that 
\begin{equation}
\label{eq:truite}
\partial_t \Phi_{B(t)}^1 =  \partial_t (\Phi_V^t \circ  \Phi_{T}^h).
\end{equation}
Then we use a classical formula, proven in Lemma \ref{lem:jolie_formule} of the Appendix, to get that
$$
\partial_t  \Phi_{B(t)}^1  = -i \Big(  \nabla \int_0^1 (\partial_{t} B(t)) \circ \Phi_{B(t)}^{-\tau} \mathrm{d}\tau \Big) \circ \Phi_{B(t)}^1.
$$
Moreover, by definition, we note that we have
$$
\partial_t (\Phi_V^t \circ  \Phi_{T}^h )= -i( \nabla V )\circ \Phi_V^t \circ  \Phi_{T}^h.
$$
Then, using that $\Phi_{B(t)}^1 =  \Phi_V^t \circ  \Phi_{T}^h$, we deduce that \eqref{eq:truite} rewrites
$$
\int_0^1 (\partial_{t} B(t)) \circ \Phi_{B(t)}^{-\tau} \mathrm{d}\tau = V.
$$
We recall that for any Hamiltonian $K,L$, we have
$$
K \circ \Phi_L^{-t} = e^{t\mathrm{ad}_{L}} K. 
$$
Note that this last relation simply comes from the formula $\partial_t K \circ \Phi_L^{-t} = \{ L, K\} \circ \Phi_L^{-t}  $.
Thus we deduce that 
$$
\varphi(\mathrm{ad}_{B(t)}) \partial_{t} B(t)  = V
$$
where $\varphi \in C^\infty(\mathbb{R};\mathbb{R})$ denotes the function defined by
$$
\varphi(X) := \frac{e^{X}-1}X. 
$$
Then, we use that by construction $B(0) = hT$ to deduce that
$$
\varphi(\mathrm{ad}_{hT}) \partial_{t} B(t)  = V - \big( \varphi(\mathrm{ad}_{B(t)})-\varphi(\mathrm{ad}_{B_0}) \big) \partial_{t} B(t).
$$
Finally, realizing the Taylor expansion in $u$ of these Hamiltonians and by grouping terms according to their degree of homogeneity, we get the following system of cohomological equations 
\begin{equation}
\label{eq:the_system_bis}
\forall n\geq 1, \quad \varphi(\mathrm{ad}_{hT}) \partial_t B_{n} = P_n -  \sum_{k\geq 0} \frac{1}{(k+1)!}  \sum_{m=1}^{n-1}    \sum_{0< i_1+\cdots + i_k = n-m }\mathrm{ad}_{B_{i_1}} \cdots \mathrm{ad}_{B_{i_k}} \partial_t B_{m} 
\end{equation}
where $V = \sum_{n\geq 1} P_n$ denotes the Taylor expansion of $V$ (see Lemma \ref{lem:def_Taylor_P}). 

\medskip

We note that the system \eqref{eq:the_system_bis} is triangular: the right hand side only depends on $B_1,\cdots,B_{n-1}$. So to solve it, it remains to invert $\varphi(\mathrm{ad}_{hT})$ (which will be done in Lemma \ref{lem:inv_phi} below).

\subsection{The cohomological equations}
\label{sub:coho}

We do not expect the solution of \eqref{eq:the_system_bis} to be convergent (i.e. that $\sum_{n\geq 0} B_n$ converges even if $V$ is polynomial). So we introduce a truncated version of \eqref{eq:the_system_bis}.
More precisely, we set $B_0 = hT$ and we consider the system of cohomological equations
\begin{equation}
\label{eq:the_system}
\varphi(\mathrm{ad}_{hT}) \partial_t B_{n} = P_n - \underbrace{ \sum_{k\geq 0} \frac{1}{(k+1)!} \overbrace{ \sum_{m=1}^{n-1}    \sum_{\substack{0< i_1+\cdots + i_k = n-m \\ 0\leq i_1,\cdots, i_k \leq r}}\mathrm{ad}_{B_{i_1}} \cdots \mathrm{ad}_{B_{i_k}} \partial_t B_{m} }^{:=K_{n,k}}}_{:=K_n}
\end{equation}
with initial condition $B_n(0) = 0$ and
where $n\in \{1,\cdots,r\}$, $t\in [0,h]$, $P=P_1+\cdots+P_r$ denotes the Taylor expansion of $V$ at order $r+2$, the unknowns $B_n$ are homogeneous polynomials of degree $n+2$ depending polynomially on $t$ (i.e. $t\mapsto B_n(t) \in \mathscr{H}^{n+2}$ is a polynomial).
We note that the system of cohomological equations \eqref{eq:the_system} is actually \emph{triangular} because  for $n\leq r$, $K_n$ only depends on $B_1,\cdots,B_{n-1}$.

First, to ensure that the system \eqref{eq:the_system} makes sense, we have to prove that, the polynomials $B_1,\cdots,B_r$ being given, the sum defining $K_n$ converges and defines a polynomial with respect to $u$ and $t$.
\begin{lemma} \label{lem:on_estime_kn}
Assume that for all $i\in \{1,\cdots,r\}$, $t\mapsto B_i(t) \in \mathscr{H}^{i+2}$ is a polynomial of degree smaller than or equal to $i$ vanishing in $t=0$. Then for all $n\geq 1$, the sum defining $K_n$ in \eqref{eq:the_system} converges\footnote{in the space of polynomials of $t$ of degree $\leq n$ and taking values in $\mathscr{H}^{n+2}$} and defines a polynomial $t\mapsto K_n(t) \in \mathscr{H}^{n+2}$ of degree  smaller than or equal to $n$. Moreover, for all $n\geq 1$ and $k\geq 0$ we have the estimates
\begin{equation}
\label{eq:lestime_quon_veut_presque}
\| K_{n,k}(t) \|_{L^\infty(0,h;\mathscr{H}^{n+2})}  \lesssim  t\, C^{ k+n} (n+2)^{k+1} (r+2)^k \big( 1+\sup_{1\leq a \leq\min(n-1,r)}  {\Vert  \partial_t B_{a} \Vert}_{L^\infty(0,h;\mathscr{H}^{a+2})} \big)^n 
\end{equation}
and
\begin{equation}
\label{eq:lestime_quon_veut}
\| K_{n}(t) \|_{L^\infty(0,h;\mathscr{H}^{n+2})}  \lesssim  t\, C^{nr}  \big( 1+\sup_{1\leq a \leq\min(n-1,r)}  {\Vert  \partial_t B_{a} \Vert}_{L^\infty(0,h;\mathscr{H}^{a+2})} \big)^n 
\end{equation}
where $C>0$ is a numerical constant.
\end{lemma}\begin{proof} First, we note that thanks to Proposition \ref{1209}, the finite sum $K_{n,k}$
defines a polynomial $t\mapsto K_{n,k}(t) \in \mathscr{H}^{n+2}$ of degree $n+2$ with respect to $u$ and of degree smaller than or equal to $n$ with respect to $t$. Such polynomials live in a finite dimensional space, so it suffices to prove the absolute convergence of the series defining $K_n$ in $L^\infty(0,h;\mathscr{H}^{n+2})$ to ensure that $t\mapsto K_n(t) \in \mathscr{H}^{n+2}$ is a well-defined polynomial of degree  smaller than or equal to $n$.

We fix $t\in [0,h]$. We set $\| B_0 \|_{\widetilde{\mathscr{H}}} := 2^{-1} h\omega_{K/2}$ and $\| B_i(t) \|_{\widetilde{\mathscr{H}}} := \| B_i(t) \|_{\mathscr{H}}$ if $i\geq 1$. Omitting many $t$ to lighten the notations and applying Proposition \ref{1209} and Lemma \ref{poissonquad} to estimate the Poisson brackets, we have
\begin{equation*}
\begin{split}
\|K_{n,k} \|_{\mathscr{H}}
\leq&  \sum_{m=1}^{n-1}4^k(n+2)^{k}(r+2)^k    \sum_{\substack{0< i_1+\cdots + i_k = n-m \\ 0\leq i_1,\cdots, i_k \leq r}} {\Vert  B_{i_1} \Vert}_{\widetilde{\mathscr{H}}}  \cdots {\Vert  B_{i_k} \Vert}_{\widetilde{\mathscr{H}}}     {\Vert \partial_t B_{m} \Vert}_{\mathscr{H}} \\
\leq & \sum_{m=1}^{n-1} 4^k(n+2)^{k} (r+2)^k \sum_{p=1}^{k} \frac{k!}{p!(k-p)!} \| B_0 \|_{\widetilde{\mathscr{H}}}^{k-p} \\
&\times  \sum_{\substack{0< i_1+\cdots + i_p = n-m \\ 1\leq i_1,\cdots, i_p \leq r}} {\Vert  B_{i_1} \Vert}_{\mathscr{H}} \cdots {\Vert  B_{i_p} \Vert}_{\mathscr{H}}    {\Vert \partial_t B_{m} \Vert}_{\mathscr{H}}.
\end{split}
\end{equation*}
Then observing that, since $1\leq i_1,\cdots,i_p$, we have $p\leq n-m\leq n$ and since the polynomials $B_i$ vanish in $t=0$, they satisfy 
$$
\| B_i(t) \|_{\mathscr{H}} \leq t \| \partial_t B_i \|_{L^\infty(0,h;\mathscr{H}^{i+2})} \leq \| \partial_t B_i \|_{L^\infty(0,h;\mathscr{H}^{i+2})} ,
$$ 
we deduce (using stars and bars combinatorics) that $ \|K_{n,k} \|_{\mathscr{H}}$ is bounded by
$$
 \sum_{m=1}^{n-1} 4^k(n+2)^{k} (r+2)^k  \sum_{p=1}^{k} \frac{k! \, t}{p!(k-p)!} \| B_0 \|_{\widetilde{\mathscr{H}}}^{k-p}  2^{2n} \big( 1+\sup_{1\leq a \leq\min(n-1,r)}  {\Vert  \partial_t B_{a} \Vert}_{L^\infty(0,h;\mathscr{H}^{a+2})} \big)^n.
$$
Then using the binomial formula and that, thanks to the CFL condition \eqref{eq:CFL}, we have that $\| B_0 \|_{\widetilde{\mathscr{H}}}\leq \pi$, we get, as expected, a numerical constant $C>0$ such that 
$$
 \|K_{n,k} \|_{\mathscr{H}}   \lesssim t\, C^{k+n} (n+2)^{k+1} (r+2)^k \big( 1+\sup_{1\leq a \leq\min(n-1,r)}  {\Vert  \partial_t B_{a} \Vert}_{L^\infty(0,h;\mathscr{H}^{a+2})} \big)^n.
$$
By definition of the exponential, the estimate \eqref{eq:lestime_quon_veut} follows directly.
\end{proof}

Then, we have to invert the operator $\varphi(\mathrm{ad}_{hT})$ uniformly with respect to $h$ and $K$.
\begin{lemma} \label{lem:inv_phi}
For all $n\leq r$, the operator $\varphi(\mathrm{ad}_{hT}) : \mathscr{H}^{n+2} \to \mathscr{H}^{n+2}$ is invertible and for all $Q\in \mathscr{H}^{n+2}$, we have the bound
\begin{equation}
\label{eq:bound_inv}
\| (\varphi(\mathrm{ad}_{hT}))^{-1} Q \|_{\mathscr{H}} \lesssim_{\delta} 1.
\end{equation}
\end{lemma}
\begin{proof}
By Lemma \ref{poissonquad}, we know that
$$
(\varphi(\mathrm{ad}_{hT}) Q)_{j}^\sigma = \varphi(-i \Omega_{j,\sigma}) Q_{j}^\sigma
$$
where  $\Omega_{j,\sigma} = \sigma_1\omega_{j_1} + \cdots +\sigma_{n+2}\omega_{j_{n+2}} $. Observing that thanks to the CFL condition \eqref{eq:CFL}, 
$$
|\varphi(-i \Omega_{j,\sigma})| \gtrsim_{\delta} 1,
$$
we deduce directly that $\varphi(\mathrm{ad}_{hT}) $ is invertible and satisfies \eqref{eq:bound_inv}.
\end{proof}

Finally, since the system \eqref{eq:the_system} is triangular, we deduce the following proposition of the two previous lemmas and Lemma \ref{lem:def_Taylor_P} for the uniform bounds on $P$.
\begin{proposition} \label{prop:sol_coho} There exists a unique sequence $B_1,\cdots,B_r$ of polynomials solving the system \eqref{eq:the_system} such that for all $i\in \{1,\cdots,r\}$, $t\mapsto B_i(t) \in \mathscr{H}^{i+2}$ is a polynomial of degree smaller than or equal to $i$ vanishing in $t=0$. Moreover, we have the bound (uniform with respect to the numerical parameters)
$$
\sup_{1\leq i \leq r}  {\Vert  \partial_t B_{i} \Vert}_{L^\infty(0,h;\mathscr{H}^{i+2})} \lesssim_{g,\delta,r,\rho} 1.
$$
\end{proposition}

\subsection{The modified Hamiltonian}

We define the modified Hamiltonian $H_h : \mathbb{C}^{\mathcal{N}_K} \to \mathbb{R}$ by
\begin{equation}
\label{eq:def_H_h}
H_h := h^{-1}B(h) \quad \mathrm{where} \quad B = B_0 + \cdots + B_r,
\end{equation}
the polynomials $B_n$ being those given by Proposition \ref{prop:sol_coho} solving the system of cohomological equations \eqref{eq:the_system}.
 We aim at proving the following theorem.
\begin{theorem}\label{main}
For all $s\geq 1/2$, there exists $\varepsilon_0 \gtrsim_{r,s,\rho,\delta} 1$,  such that for all $u\in \mathbb{C}^{\mathcal{N}_K}$ satisfying $\varepsilon := {\Vert u \Vert}_{H^s} \leq  \varepsilon_0$, we have 
\begin{equation}
\label{eq:backward}
{\Vert (\Phi_{P}^h \circ \Phi_{T}^h - \Phi_{H_{h}}^h)(u) \Vert}_{H^s} \lesssim_{r,s,\rho,\delta}  h^2 \varepsilon^{r+2}.
\end{equation}
\end{theorem}
\begin{remark} $\Phi_{H_{h}}^h$ denotes the Hamiltonian flow of $H_h$ at time $h$, i.e. the flow at times $h$ of the equation $i\partial_t u= \nabla H_h(u)$. Its existence for times smaller than or equal to $h$ is ensured by the uniform estimates of Proposition \ref{prop:sol_coho} and Proposition \ref{prop:flow}.
\end{remark}

Before proving Theorem \ref{main}, we rewrite, in the following proposition, the cohomological equation \eqref{eq:the_system}, in a more integrated form. It is one of the key points of the proof of Theorem \ref{main}.
\begin{proposition} \label{prop:rewrite} Setting, for all $t\in [0,h]$,
\begin{equation}
\label{eq:def_Qt}
Q_t :=  \int_0^1 (\partial_t B(t))\circ \Phi_{B(t)}^{-\nu} \mathrm{d}\nu,
\end{equation}
there exists $\varepsilon_0 \gtrsim_{r,\rho,\delta} 1$  such that, on $B_{H^{1/2}}(0,\varepsilon_0) \cap \mathbb{C}^{\mathcal{N}_K}$, we have 
\begin{equation}
\label{eq:coho_int}
Q_t = P + \sum_{n>r} K_n(t).
\end{equation}
\end{proposition}
\begin{remark} $\bullet$ As previously, $\Phi_{B(t)}^{\varsigma}$ denotes the (autonomous) Hamiltonian flow of $B(t)$ at time $\varsigma$, i.e. the flow at time $\varsigma$ of $i\partial_{\tau}u=(\nabla B(t))(u) $. Its existence is ensured by the same arguments as those ensuring the existence of $\Phi_{H_{h}}^h$. \\
$\bullet$ The uniform convergence of the series \eqref{eq:coho_int} is ensured by the bound \eqref{eq:lestime_quon_veut} on $K_n$ given by Lemma \ref{lem:on_estime_kn} and by the pointwise estimate given by Corollary \ref{cor:ev_pol}.
\end{remark}
\begin{proof}[Proof of Proposition \ref{prop:rewrite}] Let $u\in B_{H^{1/2}}(0,\varepsilon_0) \cap \mathbb{C}^{\mathcal{N}_K}$ where $\varepsilon_0$ will be chosen later small enough with respect to $r,\rho$ and $\delta$. First, we note that by definition of the flow and the Poisson bracket, we have  
$$
\partial_\nu [(\partial_t B(t))\circ \Phi_{B(t)}^{-\nu}(u)] = (\nabla (\partial_t B(t)), i \nabla B(t))_{L^2} \circ \Phi_{B(t)}^{-\nu}(u)  = \{ B(t), \partial_t B(t)  \}\circ \Phi_{B(t)}^{-\nu}(u)  .
$$ 
By iterating this relation, we deduce the following Taylor expansion for all $q>r$
$$
 (\partial_t B(t))\circ \Phi_{B(t)}^{-\nu}(u)  = \sum_{k=0}^q \frac{\nu^k}{k!} \mathrm{ad}_{B(t)}^k \partial_t B(t)(u) + R_{q,\nu}(u)
$$
where
$$
R_{q,\nu}(u) = \int_{0}^\nu \frac{(\nu-\varsigma)^q}{q!} \mathrm{ad}_{B(t)}^{q+1} \partial_t B(t) \circ \Phi_{B(t)}^{-\varsigma}(u) \mathrm{d}\varsigma.
$$
Integrating this expansion, we get that
$$
Q_t(u) = \sum_{k=0}^q \frac{1}{(k+1)!} \mathrm{ad}_{B(t)}^k \partial_t B(t)(u) + \int_0^1 R_{q,\nu}(u) \mathrm{d}\nu.
$$
Expanding the main factor of the sum, we have
\begin{equation}
\label{eq:pas_mal_comme_exp_non}
\begin{split}
\mathrm{ad}_{B}^k \partial_t B  &= \sum_{\ell=1}^r \sum_{0\leq i_1,\cdots,i_k \leq r} \mathrm{ad}_{B_{i_1}} \cdots \mathrm{ad}_{B_{i_k}} \partial_t B_\ell \\
&= \sum_{n \geq 1} \sum_{m=0}^{n-1} \sum_{i_1+\cdots + i_k=m} \mathrm{ad}_{B_{i_1}} \cdots \mathrm{ad}_{B_{i_k}} \partial_t B_{n-m} \\
&=: \sum_{n \geq 1} K_{n,k} + \sum_{n =1}^r \mathrm{ad}_{B_0}^k \partial_t B_n.
\end{split}
\end{equation}
It follows that, since $B$ solves the cohomological equation \eqref{eq:the_system}, we have\footnote{on $B_{H^{1/2}}(0,\varepsilon_0) \cap \mathbb{C}^{\mathcal{N}_K}$ with $\varepsilon_0$ small enough to ensure that the series converge.}
\begin{equation*}
\begin{split}
Q_t &=  \sum_{n =1}^r \varphi{(\mathrm{ad}_{B_0})} \partial_t B_n + \sum_{n\geq 1} K_n - \sum_{k>q} \sum_{n =1}^r \frac{\mathrm{ad}_{B_0}^k}{(k+1)!} \partial_t B_n - \sum_{n\geq 1}\sum_{k>q} \frac{K_{n,k}}{(k+1)!} +  \int_0^1 R_{q,\nu}\mathrm{d}\nu \\
&= P +  \sum_{n>r} K_n - \sum_{k>q} \sum_{n =1}^r \frac{\mathrm{ad}_{hT}^k}{(k+1)!} \partial_t B_n -\sum_{n\geq 1} \sum_{k>q} \frac{K_{n,k}}{(k+1)!}+  \int_0^1 R_{q,\nu} \mathrm{d}\nu
\end{split}
\end{equation*}
Finally it only remains to prove that the three remainder terms go to $0$ as $q$ goes to $+\infty$. This point is usually never justified (see \cite{MR2895408,MR2811583}). It does not require any new ideas or fundamental estimates, but it is quite heavy. We therefore only briefly describe what could be done\footnote{Note that we could also use some analyticity arguments but they require to be a little bit cautious due to the complex conjugations in the vector fields.}. For the first one, the operator $\mathrm{ad}_{hT}$ being diagonal (see Lemma \ref{poissonquad}), it is obvious. For the second one, using the bound \eqref{eq:lestime_quon_veut_presque} of Lemma \ref{lem:on_estime_kn} and the pointwise estimate of Corollary \ref{cor:ev_pol}, there exists constant a $C\lesssim_{r,\delta,\rho} 1$ such that
\begin{equation*}
\begin{split}
\sum_{n\geq 1} \sum_{k>q} \frac{|K_{n,k}(u)|}{(k+1)!}  &\lesssim \sum_{n\geq 1} \sum_{k>q} \frac1{(k+1)!}  C^{n+k} (n+2)^{k+1} (r+2)^{k+1}  \varepsilon_0^{n+2} .
\end{split}
\end{equation*}
To prove that it goes to $0$ as $q$ goes to $+\infty$, it suffices to note that, provided that $\varepsilon$ is small enough, it is the remainder term of a convergent series :
$$
\sum_{n\geq 1} \sum_{k\geq 0} \frac1{(k+1)!}  C^{n+k+3} (n+2)^{k+1} (r+2)^{k+1}  \varepsilon_0^{n+2} \leq \sum_{n\geq 1} (e^{C(r+2) })^{n+2} C^{n+2} \varepsilon_0^{n+2} <\infty.
$$
Finally, for the last term associated with $R_{q,\nu}$, we have to replace $ \mathrm{ad}_{B(t)}^{q+1} \partial_t B(t)$ by its expansion given by \eqref{eq:pas_mal_comme_exp_non} and to use that, thanks to Proposition \ref{prop:flow},
$ \|\Phi_{B(t)}^{-\varsigma}(u) \|_{H^{1/2}} \leq 2 \varepsilon_0$. Then we conclude using estimates similar to the ones used for the other terms.
\end{proof}

Thanks to this proposition, we are in position to prove Theorem \ref{main}.
\begin{proof}[\underline{Proof of Theorem \ref{main}}]
We are going to prove the following estimate
\begin{equation}
\label{eq:cequonvavraimentprouver}
\sup_{0\leq t \leq h} {\Vert v(t) \Vert}_{H^s} \lesssim_{r,s,\rho,\delta}  h^2 \varepsilon^{r+2} \quad \mathrm{where} \quad v(t)=(\Phi_{P}^t \circ \Phi_{T}^h - \Phi_{B(t)}^1)(u) .
\end{equation}
Note that, since $\Phi_{H_{h}}^h = \Phi_{h H_h}^1 = \Phi_{B(h)}^1$, \eqref{eq:cequonvavraimentprouver} is stronger than \eqref{eq:backward}. 

Then, we note that, by Lemma \ref{lem:jolie_formule} of the Appendix, we have\footnote{the formula begin local, the assumption of finite support is not restrictive. It is only used to get global solutions and so to avoid discussions about the time of existence of the solutions.}
\begin{equation}
\label{eq:cest_pas_rien}
\partial_t \Phi_{B(t)}^1 = - i (\nabla Q_t) \circ \Phi_{B(t)}^1
\end{equation}
where $Q_t$ is defined by \eqref{eq:def_Qt}. Thus, 
 we get
\begin{equation*}
\begin{split}
&{\Vert v(t) \Vert}_{H^s} \leq \int_0^{t}  {\Vert \partial_t  v(\tau)  \Vert}_{H^s}  \mathrm{d}\tau \\
\leq&  \int_0^{t}  {\Vert \nabla Q_\tau  \circ \Phi_{B(\tau)}^1(u) - \nabla P \circ  \Phi_{P}^\tau \circ \Phi_{T}^h(u)  \Vert}_{H^s}  \mathrm{d}\tau \\
\leq& \int_0^{t} \underbrace{ {\Vert \nabla P  \circ \Phi_{B(\tau)}^1(u) - \nabla P \circ  \Phi_{P}^\tau \circ \Phi_{T}^h(u)  \Vert}_{H^s}}_{=:E_1} + \underbrace{ {\Vert \nabla (Q_\tau - P)  \circ \Phi_{B(\tau)}^1(u) \Vert}_{H^s}}_{=:E_2}  \mathrm{d}\tau.
\end{split}
\end{equation*}
\noindent\emph{Estimate of $E_1$:} Using the uniform bounds on $B$ proven in Proposition \ref{prop:sol_coho} to apply Proposition \ref{prop:flow}, we deduce that, provided that $\varepsilon_0$ is chosen small enough, we have 
\begin{equation}
\label{eq:cest_pas_trop_grand}
\|\Phi_{B(\tau)}^1(u) \|_{H^s}\leq 2\varepsilon.
\end{equation}
Similarly, using the uniform bounds of Lemma \ref{lem:def_Taylor_P} on the homogeneous polynomials $P_n$ and applying Proposition \ref{prop:flow}, we deduce that\footnote{recall that $\Phi_{T}^h$ is an isometry on $H^s$.}
$$
\|\Phi_{P}^\tau \circ \Phi_{T}^h(u) \|_{H^s}\leq 2\varepsilon.
$$
Using these bounds to apply the mean value inequality with Corollary {\ref{difference}}, we get that, provided that $\varepsilon_0$ is small enough
$$
E_1 \leq   {\Vert  v(\tau) \Vert}_{H^s}.
$$
\noindent\emph{Estimate of $E_2$:} First, by Proposition \ref{prop:rewrite}, we have
$$
\nabla (Q_\tau - P)  = \sum_{n>r} \nabla K_n .
$$
Then by applying the vector field estimate given by Proposition \ref{vectorfield} with the uniform estimate on $K_n$ given by Lemma \ref{lem:on_estime_kn} and the estimate \eqref{eq:cest_pas_trop_grand} on $\|\Phi_{B(\tau)}^1(u) \|_{H^s}$, we get that
$$
\|\nabla (Q_\tau - P)(u)\|_{H^s}  \leq \sum_{n>r} \|\nabla K_n(\tau,u)\|_{H^s} \lesssim_{r,s,\delta,\rho} \tau \sum_{n>r} C^n \varepsilon^{n+1}
$$
where $C>0$ is a constant depending on $r,s,\delta,\rho$. It follows that provided that $\varepsilon_0$ is small enough, we have
$$
E_2 \lesssim_{r,s,\delta,\rho} \tau \varepsilon^{r+2}.
$$
\noindent \emph{Conclusion :} Putting the estimates on $E_1$ and $E_2$ together, we have proven that
$$
{\Vert v(t) \Vert}_{H^s} \lesssim_{r,s,\delta,\rho}  \int_0^t \| v(\tau) \|_{H^s} \mathrm{d}\tau + h^2 \varepsilon^{r+2}.
$$
By applying the Gr\"onwall inequality, we deduce that, as expected, $v$ satisfies the estimate \eqref{eq:cequonvavraimentprouver}.
\end{proof}

\subsection{Backward error analysis} In the previous subsection, in Theorem \ref{main}, we performed the backward error analysis of the Lie splitting method in the particular case where the nonlinearity $g$ is polynomial. We could extend this result to the non polynomial case by adding $V-P$ to $H_h$. However, it would generate some technicalities that are not useful in order to prove the almost global preservation of the harmonic actions (i.e. Theorem \ref{thm:main}). The point is that, actually, to prove Theorem \ref{thm:main}, the only important point is to have an approximation of the numerical flow by an Hamiltonian flow up to an error term of order $h\, \varepsilon^{r+2}$ (i.e. the exponent $2$ on $h^2$ in \eqref{eq:backward} is not necessary for us). As a consequence, we prove, in the following proposition, that the modified Hamiltonian $H_h$ provides directly such an approximation.

\begin{proposition}\label{main2}
For all $s\geq 1/2$, there exists $\varepsilon_0 \gtrsim_{r,s,\rho,\delta} 1$  such that for all $u\in \mathbb{C}^{\mathcal{N}_K}$ satisfying $\varepsilon := {\Vert u \Vert}_{H^s} \leq  \varepsilon_0$, we have 
\begin{equation}
\label{eq:backward_cinf}
{\Vert (  \Phi_V^{h}\circ \Phi_T^{h}- \Phi_{H_{h}}^h)(u) \Vert}_{H^s} \lesssim_{r,s,\rho,\delta}  h \, \varepsilon^{r+2}.
\end{equation}
\end{proposition}

In the previous subsection, we proved that $\Phi_{H_h}^h$ is a good approximation of $\Phi_{P}^h \circ \Phi_{T}^h$. So the main point to prove Proposition \ref{main2} will be to prove that $\Phi_{V}^h \circ \Phi_{T}^h$ is close to $\Phi_{P}^h \circ \Phi_{T}^h$. 
Before doing it, we prove some technical useful lemmas.

\begin{lemma}\label{polyfix}
Let $\nu > 1/2$ and $\psi\in H^{\nu}(\mathbb{Z}).$ There exists $C_{\nu}>0$ such that 
\[ {\Vert \psi^K \Vert}_{H^{\nu}} \leq C_{\nu}  {\Vert \psi\Vert}_{H^{\nu}}, \]
 where $\psi^K = \psi_{|\mathbb{T}_K}$ denotes the restriction of $\psi$ on the discrete torus $\mathbb{T}_K$.
\end{lemma}

\begin{proof}
We write $\psi(x)=\sum_{j\in \mathbb{Z}} \psi_j e^{ijx}$ in the Fourier basis and consider the collocation points $x_k=2\pi k/K\in \mathbb{T}_K$ on the grid. Applying a change of variable $j=\ell +aK$, we get 
\[\psi(x_k)=\sum_{j\in \mathbb{Z}} \psi_j e^{2\pi i\frac{j}{K}k} = \sum_{\ell \in \mathcal{N}_K}\sum_{a \in \mathbb{Z}} \psi_{\ell+aK} e^{2\pi i\frac{k}{K}\ell}= \sum_{\ell \in \mathcal{N}_K} \big( \sum_{a \in \mathbb{Z}} \psi_{\ell+aK} \big) e^{i\ell x_k}=  \psi^K(x_k).\]
It implies that the Fourier coefficients of $\psi^K$ are
$$
\psi^K_\ell = \sum\limits_{a \in \mathbb{Z}} \psi_{\ell+aK}.
$$ 
 Thus, taking the norm and applying the Cauchy--Schwarz inequality, we obtain
$$
{\Vert \psi^K \Vert}^2_{H^{\nu}(\mathbb{T}_K)}= \sum_{\ell \in \mathcal{N}_K} \langle \ell \rangle ^{2\nu} \bigg\vert \sum\limits_{a \in \mathbb{Z}} \psi_{\ell+aK} \bigg\vert^2
\leq \sum_{\ell \in \mathcal{N}_K} \left(\sum\limits_{a \in \mathbb{Z}} \langle \ell \rangle ^{2\nu} \langle a \rangle ^{2\nu} \vert\psi_{\ell+aK}\vert^2\right) \left( \sum\limits_{a \in \mathbb{Z}} \frac{1}{\langle a \rangle ^{2\nu}} \right)
$$
Moreover, if $a\in \mathbb{Z}\setminus \{0\}$ and $\ell \in \mathcal{N}_K\setminus \{0\}$, since  $|\ell|\leq K/2$, we note that
\begin{equation*}
\begin{split}
\langle \ell \rangle^2 \langle a \rangle^2 &\leq 4 \ell^2 a^2 \leq K^2 a^2 \leq 2 (|a|-1/2) K^2 \leq 2(\ell + aK)^2 \leq 2 \langle \ell + aK\rangle^2.
\end{split}
\end{equation*} 
We note that this last estimate holds trivially if $a=0$ or $\ell=0$. Hence, since $\nu>1/2$, there exists $C_{\nu}>0$ such that 
\[{\Vert \psi^K \Vert}^2_{H^{\nu}} \leq C_{\nu}^2 \sum_{\ell \in \mathcal{N}_K} \sum\limits_{a \in \mathbb{Z}} \langle \ell+aK \rangle ^{2\nu} \vert\psi_{\ell+aK}\vert^2 = C_{\nu}^2 \sum\limits_{j \in \mathbb{Z}} \langle j \rangle ^{2\nu} \vert\psi_{j}\vert^2 = C_{\nu}^2{\Vert \psi \Vert}^2_{H^{\nu}}. \]
\end{proof}

\begin{corollary} \label{cor:corico} For all $s>0$ and all $u\in \mathbb{C}^{\mathcal{N}_K}$ satisfying $\| u\|_{H^s} \leq 1$, we have
$$
\| \nabla (V - P)(u)\|_{H^s} \lesssim_{s,r,\rho} \| u\|_{H^s}^{r+2}
$$
\end{corollary}
\begin{proof}
First, we set
$\widetilde{q} = \Lambda^{-1/2}\phi(h\Lambda)\Re u.$ As a consequence, $P$ being the Taylor expansion of $V$ at order $r+2$, we have
\begin{equation}
\label{eq:trop_facile}
\nabla (V - P)(u) =  \Lambda^{-1/2} \phi(h\Lambda) g^{(\geq r+2)}( \widetilde{q})
\end{equation}
where $g^{(\geq r+2)}$ is the remainder term of the Taylor expansion of $g$ at order $r+1$ at the origin, i.e.
$$
g(y) =: \sum_{n=0}^{r+1} \frac{y^n}{n!} \partial_y^n g(0) + g^{(\geq r+2)}(y).
$$
Now, we have to be cautious. As discussed in Subsection \ref{sub:disc}, we identified functions on the discrete torus $ \mathbb{T}_K$ with functions on the continuous torus $\mathbb{T}$. Nevertheless, in this proof it is dangerous and it is better to distinguish them. Indeed, in \eqref{eq:trop_facile}, $\widetilde{q}$ is considered as a function on $ \mathbb{T}_K$ and $\Psi^K:=g^{(\geq r+2)}( \widetilde{q})$ is simply defined by composition, i.e.
$$
\forall x\in \mathbb{T}_K, \ \quad \psi^K(x) =  g^{(\geq r+2)}( \widetilde{q}(x)).
$$
However, denoting by $\breve{q} : \mathbb{T} \to \mathbb{R}$ the interpolation of $ \widetilde{q}$ to the continuous torus $\mathbb{T}$, i.e.
$$
\forall x\in \mathbb{T}, \quad \breve{q}(x) = \sum_{k\in \mathcal{N}_K}  \widetilde{q}_k e^{ikx},
$$
we also define a function $\psi : \mathbb{T} \to \mathbb{R}$ by composition,
$$
\forall x\in \mathbb{T}, \ \quad \psi(x) =  g^{(\geq r+2)}( \breve{q}(x)).
$$
The point is that, $\psi^K$ is the restriction of  $\psi$ to the grid, i.e. $ \psi^K = \psi_{|\mathbb{T}_K}$, but, a priori, the Fourier coefficients of $\psi$ are not supported on $\mathcal{N}_K$ so we cannot identify it with a function on the discrete torus.

\medskip

Anyway, by applying Lemma \ref{polyfix}, to $\psi$ and $\psi^K$, we get that
$$
\|\nabla (V - P)(u)\|_{H^{s+1}} \lesssim_{\rho} \| \psi^K \|_{H^{s+1/2}} \lesssim_{\rho,s} \| \psi \|_{H^{s+1/2}}.
$$
It only remains to control $\psi$. First, we note that by construction
$$
\| \breve{q} \|_{H^{s+1/2}} \lesssim_{\rho} \| u\|_{H^s}.
$$
So, since by assumption $ \| u\|_{H^s}\leq 1$ and $s>0$, we have
$$
\| \breve{q} \|_{L^\infty} \lesssim_{s} \| \breve{q} \|_{H^{s+1/2}}  \lesssim_{\rho,s} 1. 
$$
Since, moreover, $g^{(\geq r+2)}$ is a smooth function of order $r+2$ in $0$, it follows\footnote{here we use the basic fact that, if  $s>1/2$, $H^s(\mathbb{T})$ is a Banach algebra and if $u\in H^s(\mathbb{T};\mathbb{R})$ and if $f\in C^\infty_c(\mathbb{R};\mathbb{R})$ is a smooth compactly supported function then $f(u) \in H^s$ (see e.g. \cite[Thm 2.87 page 104]{BCD11}).} that as expected
$$
 \| \psi \|_{H^{s+1/2}} \lesssim_{s,r,\rho} \| \breve{q} \|_{H^{s+1/2}}^{r+2} \lesssim_{s,r,\rho} \| u \|_{H^{s}}^{r+2} .
$$
\end{proof}
Now, we are in position to prove Proposition \ref{main2}.
\begin{proof}[\underline{Proof of Proposition \ref{main2}}] For simplicity, here we use a special property of the nonlinear Klein--Gordon equation and the splitting method. We could do a more general proof but it would be more technical. Indeed, here, we have
$$
\Phi_{V}^h = \Phi_{V-P}^h \circ \Phi_{P}^h \quad \mathrm{and} \quad \Phi_{V-P}^h = \mathrm{Id} - ih\nabla (V-P).
$$
To prove these properties, it suffices to note that in the variables $(q,p)$, $P$ and $V$ depend only on $q$ and so that $q$ is a constant of motion of their flows (see equation \eqref{eq:flowVqp} for more explicit expressions).

Thus applying the triangular inequality, we get
\begin{equation*}
\begin{split}
 {\Vert (\Phi_{V}^h \circ \Phi_{T}^h - \Phi_{H_{h}}^h)(u) \Vert}_{H^s} &\leq  {\Vert (\Phi_{V}^h \circ \Phi_{T}^h -\Phi_{P}^h \circ \Phi_{T}^h)(u) \Vert}_{H^s} + {\Vert (\Phi_{P}^h \circ \Phi_{T}^h - \Phi_{H_{h}}^h)(u) \Vert}_{H^s} \\
 &= h \underbrace{ {\Vert \nabla(V-P)(\Phi_{P}^h \circ \Phi_{T}^h(u)) \Vert}_{H^s}}_{=:E_1} +\underbrace{ {\Vert (\Phi_{P}^h \circ \Phi_{T}^h - \Phi_{H_{h}}^h)(u) \Vert}_{H^s} }_{=:E_2}.
\end{split}
\end{equation*}
Theorem {\ref{main}} provides a control of $E_2$ by $h^2 \varepsilon^{r+2}$, so it suffices to focus on $E_1$.  Recalling that $\Phi_{T}^h$ is an isometry of $H^s$ and applying Proposition \ref{prop:flow} to estimate $\Phi_{P}^h$ with the uniform estimates on $P$ proven in Lemma \ref{lem:def_Taylor_P}, we deduce that, provided that $\varepsilon_0$ is small enough,
$$
\|\Phi_{P}^h \circ \Phi_{T}^h(u)\|_{H^s} \leq 2 \varepsilon.
$$
Therefore, it follows of Corollary \ref{cor:corico} that, as expected, provided that $\varepsilon_0$ is small enough, 
$$
E_1 \lesssim_{r,s,\delta}  \varepsilon^{r+2}.
$$
\end{proof}

 \section{Birkhoff normal form}\label{z3} 

In the previous section, we have proven that the numerical flow of the Lie splitting is equal, up to arbitrary high order error terms, to the flow of a modified Hamiltonian $H_h$. This Hamiltonian looks like the one of the semi-discretized equation \eqref{eq:KG_sd}. They are both nonlinear perturbations of the quadratic integrable Hamiltonian $T$. The aim in this section is to design a symplectic change of variable close to the identity such that in the new variable the nonlinear perturbation almost commute with $T$. It is a convenient way to average, up to arbitrarily high orders, the perturbative terms by the flow of $T$ while preserving the geometrical structure of the equation.
 
In this section $K\geq 1$ is a fixed integer. We are going to prove estimate uniform with respect to $K$. We do not need to consider the time-step $h$. Before stating the main result of this section, we introduce some standard definitions. First, we define the modulus of resonance.
\begin{definition}{(resonance modulus $\Omega_{j,\sigma}$)} 
For all $n\geq 1$, $j \in \mathcal{N}_K^{n+2}$,  $\sigma \in \{-1,1\}^{n+2}$, we set
$$
\Omega_{j,\sigma} := \sigma_1 \omega_{j_1} + \cdots+ \sigma_{n+2} \omega_{j_{n+2}}
$$
where $\omega_k = \sqrt{k^2 + \rho}$ are the frequencies of \eqref{N1}.
\end{definition}

\begin{remark}
The modulus of resonance $\Omega_{j,\sigma}$ had already appeared in the proof of Lemma \ref{lem:inv_phi}.
\end{remark}

Then we define a projection associated with terms we remove thanks to the change of variable. 
\begin{definition}{($\gamma-$resonant Hamiltonian and associated projection $\Pi_\gamma$)}
Let $K\geq 1$, $\gamma >0$, $n\geq 1$ and $Q\in \mathscr{H}^{n+2}$. We define $\Pi_\gamma Q\in \mathscr{H}^{n+2}$ by 
$$
\forall j \in \mathcal{N}_K^{n+2},  \forall \sigma \in \{-1,1\}^{n+2}, \quad (\Pi_{\gamma} Q )_j^\sigma = \mathds{1}_{|\Omega_{j,\sigma}|\geq \gamma}  Q_j^\sigma.
$$
If $\Pi_\gamma Q =0$, we say that $Q$ is $\gamma-$resonant.
\end{definition}

\subsection{The Birkhoff normal form theorem}
Now, we present the Birkhoff normal form theorem we use in this paper.
 \begin{theorem}\label{N102}
Let $r \geq 1, s\geq 1/2, \gamma \in (0,1), C>0$ and $Y = Y_1+\cdots+Y_r$ be a polynomial such that for all $i\in \{1,\cdots,r\}$, $Y_i \in \mathscr{H}^{i+2}$ and satisfies $\| Y_i \|_{\mathscr{H}} \leq C \gamma^{-i+1}$. \\
There exist
\begin{itemize}
\item a polynomial $\chi = \chi_1 + \cdots +\chi_r$ such that for all $i\in \{1,\cdots,r\}$, $\chi_i\in \mathscr{H}^{i+2}$ satisfies ${\Vert \chi_i \Vert}_{\mathscr{H}} \lesssim_{i,C} \gamma^{-i},$
\item a polynomial $Q = Q_1 + \cdots +Q_r$ such that for all $i\in \{1,\cdots,r\}$, $Q_i\in \mathscr{H}^{i+2}$ is $\gamma-$resonant and satisfies ${\Vert Q_i \Vert}_{\mathscr{H}} \lesssim_{i,C} \gamma^{-i+1},$
\item  a constant $\varepsilon_2 \gtrsim_{r,C,r} \gamma$ satisfying $\varepsilon_2 \leq \varepsilon_1$ where $\varepsilon_1>0$ is the constant given\footnote{applied with $Z=0$.} by Proposition {\ref{prop:flow}} to ensure the existence of $\Phi_{\chi}^{-1}$,
\end{itemize}
such that for all $u\in \mathbb{C}^{\mathcal{N}_K}$ satisfying $\| u\|_{H^s}\leq 2\varepsilon_2$, we have
\begin{equation}\label{decomposition}
 \bigg(T + \sum\limits_{i=1}^r Y_i\bigg)\circ \Phi_{\chi}^1(u) = (T+Q+R)(u)
  \end{equation}
  where the remainder term $R : \mathbb{C}^{\mathcal{N}_K} \cap B_{H^s}(0,2\varepsilon_2) \to \mathbb{R}$ is a smooth function satisfying
\[ {\Vert\nabla R(u)\Vert}_{H^s} \lesssim_{r,s,C} \gamma^{-r}{\Vert u\Vert}_{H^{s}}^{r+2}.\] 
 \end{theorem}
\begin{proof} This kind of Birkhoff normal form theorem is by now standard for Hamiltonian PDEs, we refer for example the reader to Theorem 2.15 in \cite{bernier:hal-03604980} and Theorem 2.12 in \cite{bernier:hal-04090717}. Moreover its proof is similar to the one we did to perform the backward error analysis of the Lie splitting in the previous section. Nevertheless, for completeness, we recall the outline of the proof.

First, we set $\mathcal{H} = T + Y_1+\cdots +Y_r$ and, as we did in the proof of Proposition \ref{prop:rewrite}, we perform the Taylor expansion of $\mathcal{H}\circ \Phi_\chi^{-1}$ where $\chi$ is an unknown polynomial that will be determined later. We get that
$$
\mathcal{H} \circ\Phi_\chi^{-1} = \sum_{k\geq 0} \frac1{k!} \mathrm{ad}_{\chi}^k \mathcal{H} .
$$
Expanding $\mathcal{H} $ and $\chi$, we get that
$$
\mathcal{H} \circ\Phi_\chi^{-1} = T+\sum_{\ell\geq 1} Q_\ell
$$
where $Q_\ell \in \mathscr{H}^{\ell+2}$ is the homogeneous polynomial defined by\footnote{note that $k$ is implicitly a summation index in these sums.}
\begin{align*}
Q_{\ell} &= \sum_{m_1+\cdots+m_k = \ell} \frac{1}{k!}\mathrm{ad}_{\chi_{m_1}} \cdots \mathrm{ad}_{\chi_{m_k}}T + \sum_{i+m_1+\cdots +m_k=\ell }\frac{1}{k!}\mathrm{ad}_{\chi_{m_1}} \cdots \mathrm{ad}_{\chi_{m_k}} Y_i\\
&=  \{\chi_{\ell}, T\}+\underbrace{\sum_{\substack{m_1+\cdots+m_k = \ell \\ k\geq 2}} \frac{1}{k!}\mathrm{ad}_{\chi_{m_1}} \cdots \mathrm{ad}_{\chi_{m_k}}T + \sum_{\substack{i+m_1+\cdots +m_k=\ell }}\frac{1}{k!}\mathrm{ad}_{\chi_{m_1}} \cdots \mathrm{ad}_{\chi_{m_k}} Y_i}_{:=F_{\ell}}
\end{align*}
and the implicit convention that $\chi_{\ell} = 0$ if $\ell >r$.

\medskip

To get the decomposition, it suffices to set
$$
R = \sum_{\ell > r} F_\ell
$$
and to require that for all $\ell \in \{1,\cdots,r\}$, $\Pi_{\gamma} Q_\ell =0$. Unfortunately, it provides a system of equations which is underdetermined with respect to $\chi$. So we impose extra conditions on $\chi$. Thus, the set of cohomological equations we consider writes
\begin{equation}
\forall \ell \in\{1,\cdots,r\}, \quad \Pi_\gamma Q_\ell =0 \quad \mathrm{and} \quad \Pi_\gamma \chi_\ell = \chi_\ell.
\end{equation}
which is equivalent to 
\begin{equation}
\label{eq:encore_un_systeme}
\forall \ell \in\{1,\cdots,r\}, \quad \{\chi_{\ell}, T\} = - \Pi_\gamma F_\ell
\end{equation}
with the extra condition that  $\Pi_0 \chi=\chi$.

\medskip 

This system of equations is very similar to the one we studied to perform the backward error analysis (i.e. \eqref{eq:the_system}). On the one hand it is triangular because $F_\ell$ depends only on $(\chi_k)_{k<\ell}$. On the other hand the implicit part of the equation is associated with an operator which is diagonal and so easy to invert. More precisely, thanks to Lemma \ref{poissonquad},  \eqref{eq:encore_un_systeme} rewrites in coordinates 
\begin{equation}
\label{eq:chi_tres_explicite}
 (\chi_{\ell})_j^{\sigma}= \begin{cases} \dfrac{({F}_{\ell})_j^{\sigma}}{i\Omega_{j,\sigma}} & \text{if } \vert \Omega_{j,\sigma}\vert \geq \gamma, \\
           0 & \text{otherwise}. \end{cases}
\end{equation}
Therefore, we design the polynomials $\chi_1,\cdots,\chi_r$ solving the system \eqref{eq:encore_un_systeme} by a direct induction.

\medskip

Then the main point is to prove by induction that for all $\ell\geq 1$
\begin{equation}
\label{eq:cequonprouvesurF}
\| F_\ell \|_{\mathscr{H}} \leq M_{(r+1)\wedge \ell}^\ell \, \gamma^{-\ell+1}
\end{equation}
where $(r+1)\wedge \ell = \min (r+1,\ell)$ and $M_1,\cdots,M_{r+1}$ is an increasing sequence of constants satisfying $C\leq M_1$. Note that by construction of $Q_\ell$, we have
$$
Q_\ell =(\mathrm{Id} - \Pi_{\gamma}) F_\ell, \quad \mathrm{and \ so} \quad \| Q_\ell \|_{\mathscr{H}} \leq \| F_\ell \|_{\mathscr{H}}.
$$
Similarly, note that, by construction of $\chi_\ell$ (see \eqref{eq:chi_tres_explicite}), we have $\| \chi_\ell \|_{\mathscr{H}} \leq \gamma^{-1} \| F_\ell \|_{\mathscr{H}}$. As a consequence, it suffices to prove the bound \eqref{eq:cequonprouvesurF} to get the expected estimates on $\chi$, $Q$ and $R$ (by applying the vector field estimate of Proposition \ref{vectorfield}).

\medskip

The proof of \eqref{eq:cequonprouvesurF} is similar to the one we did in Subsection \ref{sub:coho} to control the $K_n$ but is simpler because the sums defining $F_\ell$ are finite (contrary to those defining $K_\ell$). One difference is that, since in this section we do not assume any CFL condition, to estimate $\mathrm{ad}_{\chi_{m_1}} \cdots \mathrm{ad}_{\chi_{m_k}}T$, we use that $\chi$ solves the cohomological equation \eqref{eq:encore_un_systeme} and so that
$$
\mathrm{ad}_{\chi_{m_1}} \cdots \mathrm{ad}_{\chi_{m_k}}T = -\mathrm{ad}_{\chi_{m_1}} \cdots \mathrm{ad}_{\chi_{m_{k-1}}}  \Pi_\gamma F_\ell.
$$
We assume by induction that \eqref{eq:cequonprouvesurF} holds for $\ell<n$, and we recall that $\| \chi_\ell \|_{\mathscr{H}} \leq \gamma^{-1} \| F_\ell \|_{\mathscr{H}}$. Then, thanks to the Poisson bracket estimate of Proposition \ref{1209}, we get that\footnote{These estimates are simple but very rough (they produce huge constants). For refined estimates producing more reasonable constants, we refer for example to \cite{bernier:hal-03604980}.}
\begin{equation*}
\begin{split}
\| F_n \|_{\mathscr{H}} &\leq 2 \gamma^{-n+1} \sum_{\substack{m_1+\cdots +m_{k+1}=n   }}\frac{1}{k!} (n+2)^k (n\wedge r + 2)^k M_{(n-1)\wedge r}^{k+1} \\
&\leq 2^{n+1} \gamma^{-n+1} \sum_{k\geq 0}\frac{1}{k!} (n+2)^k (n\wedge r + 2)^k M_{(n-1)\wedge r}^{k+1}\\
& = 2^{n+1} \gamma^{-n+1}  M_{(n-1)\wedge r} (e^ {M_{(n-1)\wedge r} (n\wedge r + 2)} )^{n+2} \\
&\leq \gamma^{-n+1} (2 e^ {M_{(n-1)\wedge r} (n\wedge r + 2)} )^{4n} 
\end{split}
\end{equation*}
which concludes the derivation by setting for $j\in\{1,\cdots,r\}$, $M_{j+1} =  (2 e^ {M_{j} (j + 3)} )^{4} $ (and $M_1 =C$).

\end{proof}

\subsection{Non-resonance conditions} In order to deduce dynamical corollaries  of the Birkhoff normal form theorem, we need non resonance conditions.
\begin{proposition}[Lemma 2.5 of \cite{bernier2021birkhoff}] \label{prop:non_res} There exists an increasing sequence of nonnegative exponents $(\alpha_{r_*})\in \mathbb{R}_+^{\mathbb{N}^*}$ such that for almost all $\rho>0$, all $r_*\geq 1$, all $\ell \in (\mathbb{Z}^*)^{r_*}$ there exists a constant $C_{\ell,\rho}\in (0,1)$ such that for all $j\in \mathbb{N}^{r_*}$ with $ j_1 < \cdots < j_{r_*} $, we have 
\begin{equation}
\label{eq:est_pt_div}
\abs{\ell_1\omega_{j_1}+ \cdots + \ell_{r_*}\omega_{j_{r_*}}} > C_{\ell,\rho} \hspace{0.1cm}  \langle j_{1}\rangle^{-\alpha_{r_*}}
\end{equation}
where $\omega_j = \sqrt{j^2 + \rho}$ are the frequencies of the Klein--Gordon equation.
\end{proposition}
\begin{remark} When $\rho$ belongs to the full measure set given by Proposition \ref{prop:non_res}, we say that the frequencies are \emph{strongly non-resonant}.
\end{remark}
The non resonance condition \eqref{eq:est_pt_div} has been recently introduced by the second author and B. Gr\'ebert in \cite{bernier2021birkhoff}. It is well suited to deal with non-smooth solutions. It is stronger than the classical non resonance condition (used to deal with smooth solutions), which only require a polynomial control of the small divisors by the third largest index (i.e. $\langle j_{1}\rangle^{-\alpha_{r^*}}$ is replaced by $\langle j_{r^*-2}\rangle^{-\alpha_{r^*}}$ in \eqref{eq:est_pt_div}; see e.g. \cite{Bam03}). It is to prove that the frequencies are strongly non-resonant that we use that the frequencies converge to the integer (i.e. \eqref{eq:main_lim}).

\medskip

We will use the non-resonance condition in the following way.
\begin{lemma} \label{lem:si_ca_res_ca_com}Let $\rho>0$ belong to the full measure set given by Proposition \ref{prop:non_res}. Let $r\geq 1$, $k\in \mathcal{N}_K$ and set
$$
\gamma = \kappa_r \langle k \rangle^{-\alpha_r} \quad \mathrm{with} \quad \kappa = \min_{r_*\leq r } \min_{|\ell_1|+ \cdots+|\ell_{r_*}|\leq r} C_{\ell,\rho}.
$$
If $Q \in \mathscr{H}^{n+2}$ for some $n\leq r$ is a $\gamma-$resonant polynomial (i.e. $\Pi_{\gamma} Q =0$) then $Q$ commutes with the super-action $J_k$, i.e.
$$
\{J_k,Q\} = 0.
$$
\end{lemma}
\begin{proof}
 We set for all $b\in \mathcal{N}_K$,  $j\in \mathcal{N}_K^{n+2}$ and $\sigma\in\{-1,1\}^{n+2}$
$$
m_b(j,\sigma) :=  \sum_{a=1}^{n+2} \sigma_a \mathds{1}_{|j_a|=|b|}.
$$
We note that with these notations (since $\omega$ is even)
\begin{equation}
\label{eq:reec8om}
\Omega_{j,\sigma} = \sum_{ \substack{b\in \mathcal{N}_K\\ b\leq 0} }  \omega_b m_b(j,\sigma) \quad \mathrm{and} \quad \sum_{\substack{b\in \mathcal{N}_K \\ b\leq 0}} |m_b(j,\sigma)|\leq r .
\end{equation}
Then, we note that, thanks to Lemma \ref{poissonquad}, the coefficients of $\{J_k,Q\}$ are
\begin{equation}
\label{eq:tout_simple}
\{J_k,Q\}_j^\sigma = -i Q_j^\sigma m_k(j,\sigma). 
\end{equation}
Recalling that $\Omega_{j,\sigma}$ writes \eqref{eq:reec8om}, if $ m_k(j,\sigma)\neq 0$ then the term associated with the index $b=k$ is non zero and so, thanks to the non resonance condition \eqref{eq:est_pt_div}, it implies that
$$
|\Omega_{j,\sigma}| > \kappa \, \langle k \rangle^{-\alpha_{r_*}} \geq \kappa \,\langle k \rangle^{-\alpha_{r}} = \gamma.
$$
But since $Q$ is $\gamma$-resonant it implies that $Q_j^\sigma=0$. In other words, we have proven that $m_k(j,\sigma)\neq 0$ implies $Q_j^\sigma=0$. As a consequence of \eqref{eq:tout_simple}, it means that $\{J_k,Q\} = 0$.
\end{proof}

  \section{Dynamical consequences}\label{z4}
  In this section, we aim at proving the almost preservation of the harmonic actions (i.e. Theorem \ref{thm:main}). As in Section \ref{sec:back}, we fix $r\geq 1$, $\rho>0$, $\delta\in(0,\pi)$ and we assume that the numerical parameters $h,K$ satisfy the CFL condition \eqref{eq:CFL}, i.e.
\begin{equation*}
(r+2) \, h\, \omega_{K/2} \leq 2\pi -\delta.
\end{equation*}
We also fix an arbitrary constant $\Upsilon>1$. We are going to prove the almost preservation of the super-actions whenever $nh\leq \Upsilon \| u^0 \|_{H^{1/2}}^{-r}$. It is useful to introduce such a constant because we have to prove that the  super-actions are preserved whenever $nh\leq  (\|p^0\|_{L^2}+ \|q^0\|_{H^1})^{-r}$ and we only have $\|p^0\|_{L^2}+ \|q^0\|_{H^1} \sim_\rho \| u^0 \|_{H^{1/2}}$ (i.e. instead of an equality).

From subsection \ref{z1}, we will add a genericity assumption on the mass $\rho$ (i.e. we will assume later that $\rho$ belongs to the full measure set given by Proposition \ref{prop:non_res}).

\medskip

As in the introduction, we denote by $\Phi_{\mathrm{num}}^h$ the numerical flow of the mollified impulse method, i.e.
$$
\Phi_{\mathrm{num}}^h := \Phi^{h/2}_{V}\circ \Phi^{h}_{T} \circ \Phi^{h/2}_{V}.
$$
 We denote by $P=P_1+\cdots+P_r$ the Taylor expansion of $V$ (defined by \eqref{eq:def_V}) up to order $r+2$ as defined by equation \eqref{eq:def_Pn} in Lemma \ref{lem:def_Taylor_P}. We denote by $H_h$ the modified Hamiltonian that we constructed explicitly in Section \ref{sec:back} and that is defined by \eqref{eq:def_H_h}.

  \subsection{Almost conservation of the energy and stability of the zero solution in the energy space}

In this subsection we prove that the modified Hamiltonian  $H_h$ is almost preserved by the numerical flow and that numerical solutions initially of size $\varepsilon$ in $H^{1/2}$ remain of size $\varepsilon$ for times of order $\varepsilon^{-r}$.

We begin with two technical lemmas proving basic estimates on the modified Hamiltonian $H_h$.
\begin{lemma}\label{ellipticity}
There exist $\Lambda_{\rho}>1$ depending only on $\rho$  and $\varepsilon_{4}\gtrsim_{r,\rho,\delta} 1$ such that for all $u\in \mathbb{C}^{\mathcal{N}_K}$ satisfying ${\Vert u \Vert}_{H^{1/2}} \leq \varepsilon_{4}$, we have
\begin{equation}
\label{eq:ellipti}
   \Lambda_{\rho}^{-1}  {\Vert u \Vert}_{H^{1/2}}^2 \leq H_h(u) \leq {\Lambda_{\rho}} {\Vert u \Vert}_{H^{1/2}}^2 .
  \end{equation}
\end{lemma}
\begin{proof}
We note that there exists $\Lambda_\rho>1$ such that for all $j\in \mathbb{Z}$
$$
2(\Lambda_\rho)^{-1} \langle j \rangle \leq\omega_j \leq \frac12 \Lambda_\rho \langle j \rangle.
$$
It follows of the definition of $T$ that for all $u\in \mathbb{C}^{\mathcal{N}_K}$,
$$
2(\Lambda_\rho)^{-1} \| u \|_{H^{1/2}}^2 \leq T(u) \leq  \frac12 \Lambda_\rho \| u \|_{H^{1/2}}^2.
$$
Then, recalling that $H_h = T + h^{-1}B(h)$, we have
\begin{equation}
\label{eq:pfff}
T(u) - h^{-1}(B(h))(u) \leq | H_h(u) |\leq T(u) + h^{-1}(B(h))(u).
\end{equation}
Finally, applying the pointwise estimate of Corollary \ref{cor:ev_pol} with the uniform estimates on $B(h)$ given by Proposition \ref{prop:sol_coho}, we get that, provided that $\| u\|_{H^{1/2}}$ is small enough,
$$
|(B(h))(u)|\lesssim_{r,\rho,\delta} \| u\|_{H^{1/2}}^3.
$$
Plugging this last estimate in \eqref{eq:pfff}, we get, as expected, provided that $\| u\|_{H^{1/2}}$ is small enough, the estimate \eqref{eq:ellipti}.
\end{proof}
\begin{lemma} \label{lem:accroissement}
For all $C>0$ and all  $u,v\in \mathbb{C}^{\mathcal{N}_K}$ satisfying ${\Vert u \Vert}_{H^{1/2}} \leq C$, we have
  $$
   |\mathrm{d}H_h(u)(v)| \lesssim_{C,\rho,r,\delta} \| u\|_{H^{1/2}}\| v\|_{H^{1/2}}.
   $$
\end{lemma}
\begin{proof} Recalling that $H_h = T + h^{-1}B(h)$, we have
$$
   |\mathrm{d}H_h(u)(v)|\leq |\mathrm{d}T(u)(v)| + h^{-1} |\mathrm{d}(B(h))(u)(v)|.
$$
On the one hand, recalling that $\omega_k =\sqrt{k^2+\rho}$, by Cauchy--Schwarz we have
$$
|\mathrm{d}T(u)(v)| = \Re \sum_{k\in \mathcal{N}_k} \omega_k u_k \overline{v}_k \lesssim_{\rho} \| u\|_{H^{1/2}}\| v\|_{H^{1/2}}.
$$
On the other hand, applying the vector field estimates of Proposition \ref{vectorfield} to the polynomials $B_n(h)$ enjoying the uniform bounds given by Proposition \ref{prop:sol_coho}, we get that
$$
|\mathrm{d}(B(h))(u)(v)| = |(\nabla (B(h))(u),v)_{L^2}| \leq \|\nabla (B(h))(u)\|_{H^{1/2}} \| v\|_{H^{1/2}} \lesssim_{C,\rho,r,\delta} h\| u\|_{H^{1/2}}^2 \|v\|_{H^{1/2}}.
$$

\end{proof}

First, we deduce from the backward error analysis that while the numerical solution remains of size $\varepsilon$ in $H^{1/2}$, the modified Hamiltonian $H_h$ is almost preserved by the numerical flow.  
\begin{proposition}\label{conservation} Let $C>1$. There exists $\varepsilon_3\gtrsim_{r,\rho,\delta,C}1$ such that if $\varepsilon\leq \varepsilon_3$, $T_\varepsilon\leq \Upsilon\varepsilon^{-r}$ and $(u^n)_{n\geq 0}$ is a solution of the fully discretized nonlinear Klein--Gordon equation (i.e. for all $n\geq 0$, $u^n \in \mathbb{C}^{\mathcal{N}_K}$ satisfies $u^{n+1} = \Phi_{\mathrm{num}}^h(u^n)$) which remains of size  $C\varepsilon$ in $H^{1/2}$ for times smaller than or equal to $T_\varepsilon$, i.e. such that
\begin{equation}
\label{eq:smallness_assumption}
nh \leq T_\varepsilon \implies \| u^n \|_{H^{1/2}} \leq C\varepsilon
\end{equation}
then $(H_h(u^n))_{n\geq 0} $ is almost constant in the sense that
\[ nh \leq T_\varepsilon \implies  \abs{H_h(u^n) - H_h(u^0)} \lesssim_{C,r,\rho,\delta,\Upsilon} \varepsilon^{3}. \] 
\end{proposition}

\begin{proof} The proof is divided in $2$ steps. First, we reduce the problem to the case where the numerical flow is given by the Lie splitting. Then, we prove the result for the Lie splitting. 

\medskip

\noindent \underline{\emph{Step 1 : reduction to the Lie splitting.}} First, we use that the Strang splitting is conjugated to the Lie one. More precisely, we set 
$$
v^n = \Phi_{V}^{h/2}(u^n).
$$
 It follows that the time evolution of $v$ is given by the Lie splitting, i.e.
$$
v^{n+1} = \Phi_{V}^{h/2} \circ \Phi^{h/2}_{V}\circ \Phi^{h}_{T} \circ \Phi^{h/2}_{V} (u^n) = \Phi_{V}^{h} \circ \Phi^{h}_{T} (v^n).
$$
Moreover, $v^n$ is close to $u^n$. Indeed, as explained in the proof of Proposition \ref{main2}, we have 
$$
\Phi_{V}^{h/2} = \mathrm{Id} - \frac{h}2  i \nabla V,
$$
and so using the estimate on $\nabla V$ proved in Lemma \ref{cor:corico}\footnote{in the special case $r=0$}, we have
\begin{equation}
\label{eq:jai_faim}
\| u^n - v^n \|_{H^{1/2}} = \frac{h}2 \| \nabla V(u^n) \|_{H^{1/2}} \lesssim_{\rho} h \varepsilon^2
\end{equation}
whenever $nh \leq T_\varepsilon$. It follows that, provided that $\varepsilon_3$ is small enough, 
$$
nh \leq T_\varepsilon \implies \| v^n \|_{H^{1/2}} \leq 2C\varepsilon.
$$
Therefore, applying the mean value inequality thank to the uniform estimate of Lemma \ref{lem:accroissement}, we deduce that, whenever $nh \leq T_\varepsilon$, 
\begin{equation*}
\begin{split}
 \abs{H_h(u^n) - H_h(u^0)} & \leq  \abs{H_h(v^n) - H_h(v^0)} +  \abs{H_h(v^n) - H_h(u^n)}  +  \abs{H_h(v^0) - H_h(u^0)}  \\
 &\lesssim_{\rho,r,C}  \abs{H_h(v^n) - H_h(v^0)} + h \varepsilon^3.
\end{split}
\end{equation*}
As a consequence, to conclude, it suffices to prove that $\abs{H_h(v^n) - H_h(v^0)}\lesssim_{\rho,r,C}  \varepsilon^3$.

\medskip

\noindent \underline{\emph{Step 2 : Proof for the Lie splitting.}} Let $n\geq 1$ be such that $n h \leq T_\varepsilon$. Since $H_h$ is a constant of motion of $\Phi_{H_h}$, we have
$$
H_h(v^{n}) - H_h(v^{n-1}) = H_h( \Phi_{V}^{h} \circ \Phi^{h}_{T} (v^{n-1}) ) - H_h(\Phi_{H_h}^h(v^{n-1})).
$$
Therefore, applying the mean value inequality thanks to the uniform estimate of Lemma \ref{lem:accroissement}, we get that
$$
\abs{H_h(v^{n}) - H_h(v^{n-1})} \lesssim_{C,\rho,r,\delta} \varepsilon \| \Phi_{V}^{h} \circ \Phi^{h}_{T} (v^{n-1})  - \Phi_{H_h}^h(v^{n-1}) \|_{H^{1/2}}.
$$
Thus, applying the backward error analysis estimate given by Proposition \ref{main2}, provided that $\varepsilon_3$ is small enough, we have
$$
\abs{H_h(v^{n}) - H_h(v^{n-1})} \lesssim_{C,\rho,r,\delta} h\varepsilon^{r+3} .
$$
This estimate being uniform in $n$, it follows that
$$
\abs{H_h(v^{n}) - H_h(v^{0})} \leq \sum_{k=1}^{n}\abs{H_h(v^{k}) - H_h(v^{k-1})}   \lesssim_{C,\rho,r,\delta} nh\varepsilon^{r+3}.
$$
We conclude by using that $n h \leq T_\varepsilon \leq \Upsilon \varepsilon^{-r}$.
\end{proof}

As a corollary, we deduce that the numerical solutions remains of size $\varepsilon$ for times of order $\varepsilon^{-r}$.
\begin{corollary} \label{cor:reste_petit}
There exist $\Lambda_\rho>1$ depending only on $\rho$ and $\varepsilon_5\gtrsim_{r,\rho,\delta,\Upsilon} 1$ such that if $(u^n)_{n\geq 0}$ is a solution of the fully discretized nonlinear Klein--Gordon equation, i.e.
$$
\forall n\geq 0, \ u^n \in \mathbb{C}^{\mathcal{N}_K} \quad \mathrm{and} \quad u^{n+1} = \Phi_{\mathrm{num}}^h(u^n),
$$
such that $ \| u^0 \|_{H^{1/2}} \leq \varepsilon_5$, then while $nh \leq \Upsilon\varepsilon^{-r}$, it satisfies
$$
\| u^n\|_{H^{1/2}} \leq 2\Lambda_{\rho} \| u^0 \|_{H^{1/2}}.
$$
\end{corollary}
\begin{proof} We define $\Lambda_\rho$ as the constant given by Lemma \ref{ellipticity}, we set $\varepsilon=  \| u^0 \|_{H^{1/2}}$ and we proceed by induction on $n$. More precisely, we assume that 
$$
\forall m<n, \quad \| u^m\|_{H^{1/2}} \leq 4\Lambda_{\rho} \varepsilon.
$$
 Then, proceeding as in the proof of Proposition \ref{conservation}, we get that, provided that $\varepsilon_5$ is small enough, 
$$
\| u^n\|_{H^{1/2}} - \| u^{n-1}\|_{H^{1/2}} \lesssim_{\rho} h \varepsilon^2
$$
and so that
$$
\| u^n \|_{H^{1/2}} \leq  3\Lambda_{\rho} \varepsilon.
$$
Then applying Lemma \ref{ellipticity} and Proposition \ref{conservation} with $C=3\Lambda_{\rho}$ and $T_\varepsilon = nh$, provided that $\varepsilon_5$ is small enough we have, as expected,
\begin{equation*}
\begin{split}
\| u^{n} \|_{H^{1/2}}^2 \leq \Lambda_{\rho} H_h(u^n) \leq  \Lambda_{\rho} H_h(u^0) + \Lambda_{\rho} (H_h(u^n)- H_h(u^0))
\leq \Lambda_{\rho}^2 \varepsilon^2+ C_{\rho,r,\delta,\Upsilon} \varepsilon^3 \leq (2\Lambda_\rho \varepsilon)^2
\end{split}
\end{equation*}
where $C_{\rho,r,\delta,\Upsilon} $ is a constant depending only on $\rho,r,\Upsilon$ and $\delta$.

\end{proof}

   \subsection{Almost preservation of the harmonic actions}\label{z1}
In this subsection, we prove Theorem \ref{thm:main}.

\medskip

\noindent \underline{\emph{Step $1$ : Setting and reduction to the Lie splitting.}} From now, we assume moreover that the mass $\rho>0$ belongs to the set of full measure set given by Proposition \ref{prop:non_res} ensuring that the Klein--Gordon frequencies are strongly non-resonant. We consider a solution $(u_n)_{n\geq 0}$ to the fully discretized nonlinear Klein--Gordon equation, i.e. for all $n\geq 0$, $u^n \in \mathbb{C}^{\mathcal{N}_K}$ and $u^{n+1} = \Phi_{\mathrm{num}}^h(u^n)$, whose initial datum $u^0$ is of size $\varepsilon$ in $H^{1/2}$, i.e. $\varepsilon = \| u^0\|_{H^{1/2}}$.

\medskip

As a consequence of Corollary \ref{cor:reste_petit}, we know that, provided that $\varepsilon$ is small enough, we have
$$
\| u^n\|_{H^{1/2}} \leq 2\Lambda_{\rho} \varepsilon \quad \mathrm{whenever} \quad nh\leq \Upsilon\varepsilon^{-r}.
$$
As in the proof of Proposition \ref{conservation}, in order to consider the Lie splitting instead of the Strang one, we set, for all $n\geq 0$,
$$
v^n = \Phi_{V}^{h/2}(u^n),
$$
and we recall that
$$
v^{n+1} = \Phi_{V}^{h} \circ \Phi^{h}_{T} (v^n).
$$
Moreover, recalling \eqref{eq:jai_faim}, provided that $\varepsilon$ is small enough, we have
$$
\| u^n - v^n \|_{H^{1/2}} \lesssim_{\rho} \varepsilon^2 \quad \mathrm{whenever} \quad nh\leq \Upsilon\varepsilon^{-r}.
$$
It follows that, provided that $\varepsilon$ is small enough, while $nh\leq \Upsilon\varepsilon^{-r}$, we have
$$
\| v^n\|_{H^{1/2}} \leq 3\Lambda_{\rho} \varepsilon 
$$
and for all $k \in \mathcal{N}_K$,
$$
|J_k(u^n) - J_k(u^0)| \lesssim_{\rho} |J_k(v^n) - J_k(v^0)| + \varepsilon^3.
$$
As a consequence, it suffices to prove the existence of a constant $\beta_r$ depending only on $r$ such that 
\begin{equation}
\label{eq:cequonveutsurv}
|J_k(v^n) - J_k(v^0)|\lesssim_{\rho,\delta,r} \langle k \rangle^{\beta_r} \varepsilon^3.
\end{equation}

\medskip

\noindent \underline{\emph{Step $2$ : An upper bound on $k$.}}
From now, we fix $k\in \mathcal{N}_K$ and, as in Lemma \ref{lem:si_ca_res_ca_com}, we set
$$
\gamma := \kappa \langle k \rangle^{-\alpha_r} \quad \mathrm{with} \quad \kappa := \min_{r_*\leq r } \min_{|\ell_1|+ \cdots+|\ell_{r_*}|\leq r} C_{\ell,\rho}
$$
where $C_{\ell,\rho}$ and $\alpha_r$ are the constants given by the fact that $\rho$ is non resonant (see by Proposition \ref{prop:non_res}). In order to apply the change of variable given by the Birkhoff normal form theorem, it is convenient to have the extra assumption that $\gamma \geq \varepsilon^{1/2}$. So let us prove now that the case $\gamma < \varepsilon^{1/2}$ is quite obvious. 
Indeed, if $\gamma < \varepsilon^{1/2}$ then  
$$
\varepsilon^{-1} \leq \gamma^{-2} = \kappa^{-2} \langle k \rangle^{2\alpha_r}
$$
and so  while $nh\leq \Upsilon\varepsilon^{-r}$, we have
$$
|J_k(v^n) - J_k(v^0)|\leq \| v^n\|_{L^2}^2 + \| v^0\|_{L^2}^2 \leq 18 \Lambda_{\rho}^2 \varepsilon^2 \leq (18 \Lambda_{\rho}^2\kappa^{-2})  \langle k \rangle^{2\alpha_r} \varepsilon^3. 
$$
As a consequence, from now, we assume without loss of generality that
\begin{equation}
\label{eq:lhypo_bonus}
\gamma \geq \varepsilon^{1/2}.
\end{equation}

\medskip

\noindent \underline{\emph{Step $3$ : Application of the Birkhoff normal form theorem.}} We apply Theorem \ref{N102} with $s=1/2$, $Y = h^{-1}B(h)$ and $C>0$ the constant given by the uniform estimate of Proposition \ref{prop:sol_coho}. We get the Hamiltonians $Q,R,\chi$ satisfying the estimates given in the statement of the Birkhoff normal form theorem.  We set, when $nh\leq \Upsilon \varepsilon^{-r}$,
$$
z^n := \Phi_\chi^{-1}(v^n).
$$
Note that, by Proposition \ref{prop:flow}, thanks to the lower bound \eqref{eq:lhypo_bonus}, this definition makes sense, provided that $\varepsilon$ is small enough (we proved that the flow of $-\chi$ at time $1$ exists on a ball of radius proportional to $\gamma$). Then, using that, by Proposition \ref{prop:flow}, $\Phi_\chi^{-1}$ is close to the identity, we have
$$
\| z^n - v^n\|_{H^{1/2}} \lesssim_{\rho,r,\delta} \delta^{-1} \varepsilon^2 \sim_{\rho,r,\delta} \langle k \rangle^{\alpha_r} \varepsilon^2.
$$
As a consequence, while $nh\leq \Upsilon\varepsilon^{-r}$, we have
$$
\| z^n\|_{H^{1/2}} \leq 4\Lambda_{\rho} \varepsilon 
$$
and
$$
|J_k(v^n) - J_k(v^0)| \lesssim_{\rho,r,\delta}  |J_k(z^n) - J_k(z^0)| +  \langle k \rangle^{\alpha_r} \varepsilon^3.
$$
Thus, to conclude (i.e. to prove \eqref{eq:cequonveutsurv}), it suffices to prove the existence of a constant $\beta_r$ depending only on $r$ such that 
\begin{equation}
\label{eq:cequonveutsurz}
|J_k(z^n) - J_k(z^0)|\lesssim_{\rho,\delta,r} \langle k \rangle^{\beta_r} \varepsilon^3.
\end{equation}

\medskip

\noindent \underline{\emph{Step $4$ : Conclusion.}} Now, we fix $n\geq 0$ such that $(n+1)h\leq \Upsilon \varepsilon^{-r}$. We have
\begin{equation*}
\begin{split}
&|J_k(z^{n+1}) - J_k(z^{n})| = |J_k\circ \Phi_\chi^{-1} (\Phi_V^h \circ \Phi_T^h (v^n)) - J_k(z^{n})| \\
\leq & |\underbrace{J_k\circ \Phi_\chi^{-1} (\Phi_V^h \circ \Phi_T^h (v^n)) - J_k\circ \Phi_\chi^{-1} (\Phi_{H_h}^h (v^n))}_{=:E_{\mathrm{back}}}| +|\underbrace{   J_k\circ \Phi_\chi^{-1} (\Phi_{H_h}^h (v^n)) - J_k(z^n) }_{=:E_{\mathrm{BNF}}}|.
\end{split}
\end{equation*}
On the one hand, by the mean value inequality (using the bound \eqref{eq:deriv_flow} on $\mathrm{d}\Phi_\chi^{-1}$ given by Proposition \ref{prop:flow}), provided that $\varepsilon$ is small enough, we have
$$
|E_{\mathrm{back}}| \lesssim_{r,\rho,\delta} \varepsilon \| \Phi_V^h \circ \Phi_T^h (v^n) -\Phi_{H_h}^h (v^n) \|_{H^{1/2}}.
$$
and so, applying the backward error analysis estimate of Proposition \ref{main2}, we get
$$
|E_{\mathrm{back}}| \lesssim_{\delta,\rho,r} h\varepsilon^{r+3}.
$$
On the other hand, since $\Phi_\chi^1$ is symplectic\footnote{because it is an Hamiltonian flow.}, we have 
$$
\Phi_\chi^{-1} (\Phi_{H_h}^h (v^n)) = \Phi_{H_h\circ \Phi_\chi^1}^h (z^n).
$$
It follows that
$$
E_{\mathrm{BNF}} = \int_0^h \partial_t J_k(\Phi_{H_h\circ \Phi_\chi^1}^t (z^n))  \mathrm{d}t =  \int_0^h \{ H_h\circ \Phi_\chi^1,J_k \} \circ \Phi_{H_h\circ \Phi_\chi^1}^t(z^n) \, \mathrm{d}t.
$$
Recalling that, by the Birkhoff normal form theorem,
$$
 H_h\circ \Phi_\chi^1 = T + Q_1 + \cdots + Q_r+ R
$$
where the polynomials $Q_1,\cdots,Q_r$ are $\gamma-$resonant, by Lemma \ref{lem:si_ca_res_ca_com}, thanks to the non-resonance condition and by definition of $\gamma$, we have\footnote{the fact that $\{T,J_k\}=0$ is a just direct calculation.}
$$
\{ H_h\circ \Phi_\chi^1,J_k \} = \{ R , J_k\}.
$$
Therefore, using the estimate on $\nabla R$ given by Theorem \ref{N102} and estimating $\Phi_{H_h\circ \Phi_\chi^1}^t(z^n)$ by Proposition \ref{prop:flow}, provided that $\varepsilon$ is small enough, we have
\begin{equation*}
\begin{split}
|E_{\mathrm{BNF}}| &\leq \int_0^h \|  \Phi_{H_h\circ \Phi_\chi^1}^t(z^n) \|_{H^{1/2}} \|  \nabla R(\Phi_{H_h\circ \Phi_\chi^1}^t(z^n) )\|_{H^{1/2}}  \, \mathrm{d}t \\
&\lesssim_{r,\rho,\delta} h \gamma^{-r}\varepsilon^{r+3} \sim_{r,\rho,\delta} h \langle k \rangle^{r\alpha_r}\varepsilon^{r+3} .
\end{split}
\end{equation*}
Finally, summing the error terms $E_{\mathrm{back}}$ and $E_{\mathrm{BNF}}$, we get \eqref{eq:cequonveutsurz}, i.e. if $nh \leq \Upsilon\varepsilon^{-r}$ then
$$
|J_k(z^{n}) - J_k(z^{0})|\lesssim_{\delta,\rho,r} n h \langle k \rangle^{r\alpha_r}\varepsilon^{r+3} \lesssim_{\delta,\rho,r,\Upsilon} \langle k \rangle^{r\alpha_r}\varepsilon^{3} .
$$

\section{Conclusion} 

We proved that, even at low regularity, the mollified impulse method preserves some important qualitative properties of the nonlinear Klein--Gordon equations. In particular, we proved that for small numerical solutions in the energy space, the energy and the harmonic actions are almost preserved for very long times. To prove it, first we performed the backward error analysis of the method (which directly provides the almost preservation of the energy) and, then, we put the modified equation in some kind of Birkhoff normal form. Finally, the uniform bound on the norm of the solution given by the almost preservation of the energy and the commutation property given by this normal form allowed us to deduce the almost preservation of the harmonic actions at low regularity.

\medskip

This approach  and the formalism we present is fairly robust. Adapting the dynamical consequences of the Birkhoff normal form (in the spirit of what is done in \cite{Bam03}), it would provide an alternative proof of high regularity results (such that \cite{zbMATH05323311}). We also expect it to be adaptable to other equations such as nonlinear Schr\"odinger equations or beam equations.

\medskip

Nevertheless, many questions remain open. In particular, we are very interested in the preservation of the high harmonic actions (i.e. $J_k$ with $k$ large).
Current approaches do not allow us to control their variations (for very long times) at low regularities. This contrasts with numerical observations, where these quantities seem particularly stable. We think that it would deserve in-depth numerical studies and further investigations.




 \begin{appendix}
\section{Appendix}\label{1001}

We prove a general formula about Hamiltonian systems depending on a parameter. This basic formula is well known (e.g. it is used in \cite{MR2895408,MR2811583}) but we did not find any complete proof in the literature. 
\begin{lemma} \label{lem:jolie_formule} Let $n\geq 1$, $J\in \mathbb{R}^{2n\times 2n}$ be a skew symmetric invertible matrix, $Q_{\alpha} \in \mathcal{C}^\infty_c(\mathbb{R}^{2n};\mathbb{R})$ be a family of smooth, compactly supported, real-valued functions indexed by $\alpha \in \mathbb{R}$ such that $(\alpha,x)\mapsto Q_\alpha(x) \in \mathcal{C}^\infty(\mathbb{R}^{2n}\times \mathbb{R};\mathbb{R})$. Denoting by $\Phi_{Q_\alpha}$ the Hamiltonian flow\footnote{note that since $Q_\alpha$ is finitely supported, this flow is globally well defined.} of $Q_\alpha$ (i.e. the flow of the equation $y' = J\nabla Q_\alpha(y)$), for all $\alpha \in \mathbb{R}$ and all $t\in \mathbb{R}$, we have
\begin{equation}
\label{eq:jolie_eq_pas_facile}
\partial_\alpha \Phi_{Q_\alpha}^t = \Big( J \nabla \int_0^t (\partial_{\alpha} Q_\alpha) \circ \Phi_{Q_\alpha}^{-\tau} \mathrm{d}\tau \Big) \circ \Phi_{Q_\alpha}^t.
\end{equation}
\end{lemma}
\begin{proof} For any $H\in \mathcal{C}^\infty_c(\mathbb{R}^{2n};\mathbb{R})$, we define the differential operator $\mathrm{ad}_H := J\nabla H \cdot \nabla$ acting on $C^\infty(\mathbb{R}^{2n};\mathbb{R})$. For any $P\in \mathcal{C}^\infty(\mathbb{R}^{2n};\mathbb{R})$, any $t,\alpha\in \mathbb{R}$, we set
$$
P_\alpha(t) := e^{t\mathrm{ad}_{Q_\alpha}} P := P\circ \Phi_{Q_\alpha}^t.
$$ 
Note that $(t,\alpha) \mapsto P_\alpha(t)$ is a smooth function.
These notations make sense because $P_\alpha$ is solution of the transport equation
$$
\partial_t P_\alpha = \mathrm{ad}_{Q_\alpha} P_\alpha.
$$
By taking the derivative of this equation with respect to $\alpha$, one gets
$$
\partial_t \partial_\alpha P_\alpha = \mathrm{ad}_{\partial_\alpha Q_\alpha} P_\alpha +  \mathrm{ad}_{Q_\alpha} \partial_\alpha P_\alpha.
$$
Then, by applying the Duhamel formula, one gets
\begin{equation}
\label{eq:cavabientotservir}
 \partial_\alpha P_\alpha(t) = e^{t\mathrm{ad}_{Q_\alpha}} \int_0^t e^{-\tau \mathrm{ad}_{Q_\alpha}} \mathrm{ad}_{\partial_\alpha Q_\alpha} P_\alpha(\tau) \mathrm{d}\tau .
\end{equation}
Moreover, $\Phi_{Q_\alpha}$ being symplectic, we have
\begin{equation*}
\begin{split}
J\nabla \big[ (\partial_\alpha Q_\alpha)\circ \Phi_{Q_\alpha}^{-\tau} \big] = J\big( \mathrm{d} \Phi_{Q_\alpha}^{-\tau} \big)^* [ \nabla  (\partial_\alpha Q_\alpha)] \circ \Phi_{Q_\alpha}^{-\tau}   = \big( \mathrm{d} \Phi_{Q_\alpha}^{-\tau} \big)^{-1}J  [\nabla  (\partial_\alpha Q_\alpha)] \circ \Phi_{Q_\alpha}^{-\tau} .
\end{split}
\end{equation*}
It follows that
\begin{equation*}
\begin{split}
e^{-\tau \mathrm{ad}_{Q_\alpha}} \mathrm{ad}_{\partial_\alpha Q_\alpha} P_\alpha(\tau) &= \big( \mathrm{ad}_{\partial_\alpha Q_\alpha} (P\circ \Phi_{Q_\alpha}^\tau)\big)\circ \Phi_{Q_\alpha}^{-\tau} \\
&= \big( J\nabla \partial_\alpha Q_\alpha \cdot \nabla (P\circ \Phi_{Q_\alpha}^\tau)\big)\circ \Phi_{Q_\alpha}^{-\tau} \\
&=  J[\nabla (\partial_\alpha Q_\alpha )]\circ \Phi_{Q_\alpha}^{-\tau} \cdot  (\mathrm{d}\Phi_{Q_\alpha}^{-\tau} )^{-*} \nabla P  \\
&= J\nabla \big[ (\partial_\alpha Q_\alpha)\circ \Phi_{Q_\alpha}^{-\tau} \big] \cdot \nabla P.
\end{split}
\end{equation*}
Plugging this relation in \eqref{eq:cavabientotservir}, we get
$$
 \partial_\alpha (P\circ \Phi_{Q_\alpha}^t) =  \Big[J\nabla \int_0^t  (\partial_\alpha Q_\alpha)\circ \Phi_{Q_\alpha}^{-\tau}  \mathrm{d}\tau \cdot \nabla P  \Big]\circ\Phi_{Q_\alpha}^t .
$$
Now, in the particular case where $P$ is a linear form, i.e. $P(x) = z\cdot x$ for some $z\in \mathbb{R}^{2n}$, this last formula gives
$$
z\cdot \partial_\alpha  \Phi_{Q_\alpha}^t =  \Big[J\nabla \int_0^t  (\partial_\alpha Q_\alpha)\circ \Phi_{Q_\alpha}^{-\tau}  \mathrm{d}\tau  \big] \circ\Phi_{Q_\alpha}^t \cdot z.
$$
This relation holding for all $z\in \mathbb{R}^n$, by duality, we have proven \eqref{eq:jolie_eq_pas_facile}.

\end{proof}

\end{appendix}


\begin{thebibliography}{100}


\bibitem{Charb} 
{\sc C. Abou Khalil}, 
\emph{Birkhoff normal form in low regularity for the nonlinear quantum harmonic oscillator}, 
\href{https://doi.org/10.1088/1361-6544/ada09f}{Nonlinearity}  38 015020 2025



\bibitem{Ala23}
{\sc Y. Alama Bronsard}, {\em A symmetric low-regularity integrator for the nonlinear Schr\"odinger equation,} \href{https://doi.org/10.1093/imanum/drad093}{IMA Journal of Numerical Analysis}, drad093, 2023  


\bibitem{BCD11} 
{\sc H.~Bahouri, J.-Y.~Chemin, R.~Danchin},
\emph{Fourier Analysis and Nonlinear Partial Differential Equations}, \href{https://doi.org/10.1007/978-3-642-16830-7}{Grundlehren der mathematischen Wissenschaften}, Springer-Verlag Berlin Heidelberg, 2011


\bibitem{Bam03} 
{\sc D. Bambusi}, 
\emph{Birkhoff Normal Form for Some Nonlinear PDEs}, 
\href{https://doi.org/10.1007/s00220-002-0774-4}{Commun. Math. Phys.} (2003) 234: 253.


\bibitem{BG06} 
{\sc D. Bambusi, B. Gr\'ebert}, 
\emph{Birkhoff normal form for partial differential equations with tame modulus}, 
\href{https://doi.org/10.1215/S0012-7094-06-13534-2}{Duke Math. J.} 135 (2006), no. 3, 507--567.



\bibitem{BFG13} 
{\sc D. Bambusi, E. Faou, B. Gr\'ebert}, 
\emph{Existence and stability of ground states for fully
discrete approximations of the nonlinear Schrödinger equation}, 
\href{https://dx.doi.org/10.1007/s00211-012-0491-7}{Numer. Math.} 123 (2013),
pp. 461–492





\bibitem{bernier2021birkhoff}
{\sc J. Bernier, B. Gr\'ebert}, 
\emph{Birkhoff normal forms for Hamiltonian PDEs in their energy space}, 
\href{https://doi.org/10.5802/jep.193}{Journal de l’Ecole polytechnique — Mathématiques}, Tome 9 (2022), pp. 681-745.


\bibitem{bernier:hal-03604980}
{\sc J. Bernier, B. Gr\'ebert}, 
\emph{Almost global existence for some nonlinear Schr\"odinger equations on $\mathbb{T}^d$ in low regularity}, 
\href{https://doi.org/10.5802/aif.3706}{Annales de l'Institut Fourier } (2025). 



\bibitem{BGR21}
{\sc J. Bernier, B. Gr\'ebert, G. Rivi\`ere}, \emph{Dynamics of nonlinear Klein-Gordon equations in low regularity on $\mathbb{S}^2$}, \href{https://doi.org/10.4171/AIHPC/55}{Ann. Inst. H. Poincar\'e Anal. Non Lin\'eaire}, 40 (2023), no. 5, pp. 1009–1049.




\bibitem{bernier:hal-04090717}
{\sc J.~Bernier, B.~Gr{\'e}bert, T.~Robert},  
\emph{Dynamics of quintic nonlinear Schr{\"o}dinger equations in
  $H^{2/5+}(\mathbb{T})$}, 
  \href{https://arxiv.org/abs/2305.05236}{arXiv:2305.05236}, To appear in  Transactions of the AMS (2023)





\bibitem{BS22}
{\sc Y. Bruned, K. Schratz}, \emph{Resonance based schemes for dispersive equations via decorated trees}, \href{https://doi.org/10.1017/fmp.2021.13}{Forum of Mathematics, Pi}, Volume 10 , 2022 , e2.


\bibitem{MR2222808}
{\sc B. Cano}, \emph{Conserved quantities of some Hamiltonian wave equations after full discretization}, \href{https://doi.org/10.1007/s00211-006-0680-3}{Numer. Math.}. \textbf{103}, 197-223 (2006)

\bibitem{demande_reviewer}
{\sc B. Cano}, \emph{Conservation of invariants by symmetric multistep cosine methods for second-order partial differential equations}, \href{https://doi.org/10.1007/s10543-012-0393-1}{ Bit Numer Math} 53, 29–56 (2013).




\bibitem{zbMATH05323311}
{\sc D. Cohen, E. Hairer, C. Lubich},
\emph{Conservation of energy, momentum and actions in numerical discretizations of non-linear wave equations}, \href{https://doi.org/10.1007/s00211-008-0163-9}{Numer. Math.}, 110, 113–143 (2008). 




\bibitem{MR2366141}
{\sc D. Cohen, E. Hairer, C. Lubich},
\emph{Long-Time Analysis of Nonlinearly Perturbed Wave Equations Via Modulated Fourier Expansions}, \href{https://doi.org/10.1007/s00205-007-0095-z}{Arch Rational Mech Anal},    187, 341–368 (2008).


\bibitem{MR2895408}
{\sc E. Faou}, \emph{Geometric numerical integration and Schrödinger equations}, European Mathematical Society (EMS), Zürich, 2012 



\bibitem{MR2811583}
{\sc E. Faou, B. Grébert},
\emph{Hamiltonian interpolation of splitting approximations for nonlinear PDEs}, \href{https://doi.org/10.1007/s10208-011-9094-4}{Found. Comput. Math.}. \textbf{11}, 381-415 (2011)

\bibitem{MR2570074}
{\sc E. Faou, B. Grébert, E. Paturel},
\emph{Birkhoff normal form for splitting methods applied to semilinear Hamiltonian PDEs. I. Finite-dimensional discretization},
 \href{https://doi.org/10.1007/s00211-009-0258-y}{Numer. Math.}. \textbf{114}, 429-458 (2010)


\bibitem{GSS98}
{\sc B. Garc\'ia-Archilla, J. M. Sanz-Serna, R. D. Skeel}, 
\emph{Long-Time-Step Methods for Oscillatory Differential Equations},
 \href{https://doi.org/10.1137/S1064827596313851}{SIAM Journal on Scientific Computing},  1998 20:3, 930-963 



\bibitem{MR3712186}
{\sc L. Gauckler},
\emph{Numerical long-time energy conservation for the nonlinear Schr\"odinger equation},
 \href{https://doi.org/10.1093/imanum/drw057}{IMA J. Numer. Anal.}. \textbf{37}, 2067-2090 (2017)

\bibitem{GL10}
{\sc L. Gauckler, C. Lubich},
\emph{Splitting Integrators for Nonlinear Schrödinger Equations Over Long Times}, \href{https://doi.org/10.1007/s10208-010-9063-3}{Found Comput Math} 10, 275–302 (2010). 




\bibitem{MR2413147}
{\sc E. Hairer, C. Lubich},
\emph{Spectral semi-discretisations of weakly non-linear wave equations over long times},
 \href{https://doi.org/10.1007/s10208-007-9014-9}{Found. Comput. Math.}. \textbf{8}, 319-334 (2008)

\bibitem{MR2840298}
{\sc E. Hairer, C. Lubich, G. Wanner},
\emph{Geometric numerical integration}. (Springer, Heidelberg,2010), Structure-preserving algorithms for ordinary differential equations, Reprint of the second (2006) edition




\bibitem{MS23}
{\sc G. Maierhofer, K. Schratz}, 
\emph{Bridging the gap: symplecticity and low regularity in Runge-Kutta resonance-based schemes,}
 \href{https://doi.org/10.48550/arXiv.2205.05024}{arXiv:2205.05024} (2023)



\bibitem{MR4567995}
{\sc B. Li, K. Schratz, F. Zivcovich},
\emph{A second-order low-regularity correction of Lie splitting for the semilinear Klein-Gordon equation},
 \href{https://doi.org/10.1051/m2an/2022096}{ESAIM Math. Model. Numer. Anal.}. \textbf{57}, 899-919 (2023)



\bibitem{ORS23}
{\sc A. Ostermann, F. Rousset, K. Schratz}, {\em Fourier integrator for periodic NLS: low regularity estimates via discrete Bourgain spaces,} \href{https://doi.org/10.4171/JEMS/1275}{J. Eur. Math. Soc.}, 25 (2023), no. 10, pp. 3913–3952

\bibitem{OS18}
{\sc A. Ostermann, K. Schratz}, {\em Low regularity exponential-type integrators for semilinear Schr\"odinger equations,} \href{https://doi.org/10.1007/s10208-017-9352-1}{Found Comput Math}, 18, 731–755 (2018). 

\bibitem{RS21}
{\sc F. Rousset, K. Schratz}, {\em A general framework of low-regularity integrators,} \href{https://doi.org/10.1137/20M1371506}{SIAM J. Numer. Anal.}, 59, 3, 1735-1768, 2021.


\bibitem{WZ22}
{\sc Y. Wang, X. Zhao}, {\em A symmetric low-regularity integrator for nonlinear Klein-Gordon equation,} \href{https://doi.org/10.1090/mcom/3751}{Math. Comp.}, 91 (2022), 2215-2245 





\end{thebibliography}
\end{document}